\documentclass[a4paper,10pt]{article}
\usepackage[all]{xy}
\usepackage[T1]{fontenc}
\usepackage{graphicx}
\usepackage[utf8]{inputenc}
\usepackage{amsmath,amsthm}
\usepackage{amsfonts,amssymb}
\usepackage{bigdelim}
\usepackage{multirow}
\usepackage{enumerate}
\usepackage{float}
\usepackage{pstricks}
\usepackage[ae]{}
\usepackage{authblk}
\usepackage{pst-plot,pst-intersect,mathtools}
\usepackage{natbib}
\usepackage[draft]{pgf}
\usepackage{listings}
\usepackage{xcolor}
\newcommand{\E}{\mathbb{E}}
\newcommand{\N}{\mathbb{N}}

\newcommand{\R}{\mathbb{R}}

\newcommand{\Pb}{\mathbb{P}}
\newcommand{\tpr}{t^{\prime}}
\newcommand{\spr}{s^{\prime}}
\newcommand{\mt}{\mathbf{t}}
\newcommand{\ms}{\mathbf{s}}

\newcommand{\ve}{\varepsilon}

\newcommand{\wtb}{\widetilde{\beta}}

\def\={{\;\mathop{=}\limits^{\text{(law)}}\;}}

\newtheorem{theorem}{Theorem}[section]
\newtheorem{prop}[theorem]{Proposition}

\newtheorem{defi}[theorem]{Definition}
\newtheorem{corol}[theorem]{Corollary}

\theoremstyle{definition}
\newtheorem{rem}[theorem]{Remark}

\newtheorem{exa}[theorem]{Example}

\numberwithin{equation}{section}

\usepackage{fancyhdr}
\date{}

\usepackage[english]{babel}

\title{Mean residual life processes and associated submartingales}
\author[]{Antoine-Marie Bogso
}
\affil[]{University of Yaounde I, Department of Mathematics, P.O. Box 812 Yaounde, Cameroon\\ Emails: ambogso@uy1.uninet.cm, ambogso@gmail.com,\\ Phone: (+237)652620452}
\date{}

\begin{document}
 
\maketitle
\begin{abstract}We use Madan-Yor's argument to construct associated submartingales to a class of two-parameter processes that are ordered by the increasing convex dominance. This class includes processes whose integrated survival functions are multivariate totally positive of order 2 (MTP$_2$). We prove that the integrated survival function of an integrable two-parameter process is MTP$_2$ if and only if it is totally positive of order 2 (TP$_2$) in each pair of arguments when the remaining argument is fixed. This result can not be deduced from known results since there are several two-parameter processes whose integrated survival functions do not have interval support. Since the MTP$_2$ property is closed under several transformations, it allows to exhibit many other processes having the same total positivity property.\\   
\text{}
  \\
\textbf{keywords: }Cox-Hobson algorithm, Incomplete Markov processes, MRL ordering, Two-parameter submartingales, Total positivity.\\
 \textbf{subclass MSC:} 60E15, 60G44, 60J25, 32F17.
\end{abstract}

\section{Introduction}
\label{intro}
The connection between mean residual life (MRL) ordering and the martingale property originated to Madan and Yor \citep{MY02}. Using the Az\'ema-Yor solution to the Skorokhod embedding problem, these authors provided, for every constant-mean  single-parameter mean residual life (MRL) process, an explicit associated martingale, i.e. a martingale with the same one-dimensional marginals. They also exhibited many examples of  single-parameter MRL processes. Several other examples of single-parameter MRL processes may be found in \citep{Bo15,HPRY11,LYY13}. Since single-parameter MRL processes are ordered by the increasing convex dominance, the existence of an associated martingale to a given single-parameter constant-mean MRL process follows directly from a remarkable Kellerer's result which states that processes ordered by the increasing convex dominance have submartingale marginals. Moreover, this submartingale may be chosen Markovian. In particular, for single-parameter peacocks, i.e. for constant-mean single-parameter processes which are ordered by the increasing convex dominance, the Kellerer's theorem yields the existence of an associated Markovian martingale. But the Madan-Yor's argument does not apply to non-MRL peacocks. Therefore, other Skorokhod embedding solutions have been used to construct explicitly an associated martingale to a given single-parameter peacock (see e.g. \citep[Chapter 7]{HPRY11} and \citep{KTT17}). We refer to \citep{BCH17} and \citep{Ob04},  where numerous solutions to the Skorokhod embedding problem are presented. A motivation of the Skorokhod embedding approach to associate a martingale to a given peacock is the Dambis-Dubins-Schwarz theorem which states that every single-parameter martingale is a time-changed one-dimensional Brownian motion. The idea of Madan and Yor extends to the whole family of MRL processes. Indeed, as the MRL ordering is preserved by translation, the generalization of the Az\'ema-Yor embedding algorithm to non-centered target distributions due to Cox and Hobson \citep{CH06} yields an explicit associated submartingale to every single-parameter MRL process. Moreover, this submartingale is Markovian when the marginals of the MRL process have densities. In the case of two-parameter processes, the Madan-Yor argument is still valid. In particular, every two-parameter MRL process is associated to a two-parameter submartingale. On the contrary, there does not yet exist a counterpart of Kellerer's theorem for two-parameter processes. Recently, Juillet \citep{Ju16} exhibited counterexamples which show that Kellerer's theorem fails in the two-parameter case. He then answered an open question formulated in \citep{HPRY11} on the existence of an associated two-parameter martingale to a given two-parameter peacock. Juillet also provided a family of centered two-parameter peacocks for which there exists an associated two-parameter martingale. Precisely, he introduced the notion of {\it diatomic convex ordering} and constructed explicitly an associated martingale to every two-parameter process which is non-decreasing in the diatomic convex ordering. However, one may observe that centered diatomic convex ordered processes are convex combinations of centered diatomic MRL processes, and that Juillet exploits the following argument: since every centered diatomic MRL process is associated to a martingale measure, then a convex combination of centered diatomic MRL processes is associated to a martingale measure equals to a convex combination of the corresponding associated martingale measures. We apply this idea to other two-parameter processes. Precisely, we provide new families of two-parameter MRL processes, and then we construct explicitly associated submartingales to several non-MRL processes. Among MRL processes, there are processes with multivariate totally positive of order 2 (MTP$_2$) integrated survival functions (see Paragraph \ref{ssec:MTP2} for the definiton of the MTP$_2$  property). We prove and apply the following result: the integrated survival function of an integrable two-parameter process is MTP$_2$ if and only if it is totally positive of order 2 (TP$_2$) in each pair of arguments when the remaining argument is fixed. According to \citep[Proposition 2.1]{KaR80} (see also \citep[Proposition 3.5]{Fal17}), such a result holds for nonnegative functions which have interval support. We mention that the terminology ``{\it having interval support }'' is borrowed from \citep[Page 7]{Fal17}. Note that several integrated survival functions do not have interval support and, as a consequence, the result of Karlin and Rinot does not apply. The MTP$_2$ property of the integrated survival functions of certain two-parameter processes allows to generate many other processes with MTP$_2$ integrated survival functons. Indeed, the MTP$_2$ property is closed under several convex transformations.  

\subsection{Two-parameter mean residual life processes}
We define a concept of two-parameter mean residual life (MRL) process, and, using Madan-Yor's argument, we show that the Cox-Hobson algorithm provides an associated submartingale to a given two-parameter MRL process.

Let $\R_+^2$ denote the first quadrant of the coordinate plane. We endow $\R_+^2$ with the following usual partial ordering: for every $\ms=(s,\spr)$ and $\mt=(t,\tpr)$ in $\R_+^2$, $\ms\leq\mt$ means $s\leq t$ and $\spr\leq\tpr$. 
\begin{defi}
We call two-parameter mean residual life (MRL) process a family of integrable probability measures $\left(\mu_{\mt},\mt\in\R_+^2\right)$ such that the corresponding family of Hardy-Littlewood functions $\left(\Psi_{\mu_\mt},\mt\in\R_+^2\right)$ defined on the real line $\R$ by
$$
\Psi_{\mu_\mt}(x)=\left\{
\begin{array}{ll}
\dfrac{1}{\mu_\mt([x,+\infty[)}\displaystyle\int_{[x,+\infty[}y\mu_\mt(dy)&\text{if }x<r_{\mu_\mt},\\
x&\text{if }x\geq r_{\mu_\mt},
\end{array}
\right.
$$
where $r_{\mu_\mt}=\inf\{z\in\R;\,\mu_\mt([z,+\infty[)=0\}$, is pointwise non-decreasing, i.e. for every $x\in\R$ and every $\ms\leq\mt$, $\Psi_{\mu_\ms}(x)\leq\Psi_{\mu_\mt}(x)$.
\end{defi}
\begin{rem}\label{rem:MRLPsiRLmu}\text{}
\item[1.]Let $\Psi_{\nu}$ denote the Hardy-Littlewood function of an integrable probability measure $\nu$. Then $\Psi_{\nu}$ is left-continuous, non-decreasing and, for every $x\in\R$, $\Psi_{\nu}(x)\geq x$.
Moreover, one has $\lim\limits_{x\to-\infty}\Psi_{\nu}(x)=\int_{\R}y\nu(dy)$. As a consequence, if $\left(\mu_\mt,\mt\in\R_+^2\right)$ is a MRL process, then $\mt\longmapsto\int_{\R}y\mu_\mt(dy)$ is non-decreasing.
\item[2.]Let $\left(\mu_\mt,\mt\in\R_+^2\right)$ be a family of integrable probability measures. For every $\mt\in\R_+^2$, let $r_{\mu_\mt}$  denote the upper bound of the support of $\mu_\mt$. Precisely,
$$
r_{\mu_\mt}=\inf\{x\in\R;\,\mu_\mt([x,+\infty[)=0\}.
$$
Observe that $r_{\mu_\mt}$ rewrites as follows (see e.g. \citep[Chapter VI, Lemma 5.1]{ReY99}):
$$
r_{\mu_\mt}=\inf\{x\in\R;\,\Psi_{\mu_\mt}(x)=x\}.
$$
If $\left(\mu_\mt,\mt\in\R_+^2\right)$ is MRL ordered, then $\mt\longmapsto r_{\mu_\mt}$ is non-decreasing. Indeed, for $\ms\leq\mt$, if $r_{\mu_\mt}=+\infty$, then $r_{\mu_\ms}\leq r_{\mu_\mt}$; otherwise, one has 
$$
r_{\mu_\mt}\leq\Psi_{\mu_\ms}(r_{\mu_\mt})\leq\Psi_{\mu_\mt}(r_{\mu_\mt})=r_{\mu_{\mt}}
$$
from which we deduce that $r_{\mu_\ms}\leq r_{\mu_\mt}$.
\end{rem}
In the two-parameter case, the MRL ordering is related to the following submartingale property (see e.g. \citep[Section 1]{Mi83}).
\begin{defi}
Let $(\Omega,\mathcal{F},\Pb)$ be a complete probability space, and let $(\mathcal{F}_\mt,\mt\in\R_+^2)$ be a filtration of $\mathcal{F}$, i.e. a family of non-decreasing sub-sigma-algebras of $\mathcal{F}$. A process $\left(X_\mt,\mt\in\R_+^2\right)$ is said to be $(\mathcal{F}_\mt)$-adapted if $X_\mt$ is $\mathcal{F}_\mt$-measurable for every $\mt\in\R_+^2$. An integrable $(\mathcal{F}_\mt)$-adapted process $\left(X_\mt,\mt\in\R_+^2\right)$ is a $(\mathcal{F}_\mt)$-submartingale if, for every $\ms\leq\mt$, $X_\ms\leq\E[X_\mt\vert\mathcal{F}_\ms]$. 
\end{defi}
\begin{rem}
Let $\left(X_\mt,\mt\in\R_+^2\right)$ be an integrable and adapted process with respect to some filtration $\left(\mathcal{F}_\mt,\mt\in\R_+^2\right)$.
\item[1.]If $\left(X_\mt,\mt\in\R_+^2\right)$ is a $(\mathcal{F}_\mt)$-submartingale with a constant mean, then it is a $(\mathcal{F}_\mt)$-martingale, i.e., for every $\ms\leq\mt$, $\E[X_\mt\vert\mathcal{F}_\ms]=X_\ms$.
\item[2.]Suppose that $\left(\mathcal{F}_\mt,\mt\in\R_+^2\right)$ is the natural filtration of $\left(X_\mt,\mt\in\R_+^2\right)$. Then $\left(X_\mt,\mt\in\R_+^2\right)$ is a $(\mathcal{F}_\mt)$-submartingale if, and only if, for every positive integer $n$, every $\ms_1\leq\cdots\leq\ms_n\leq\ms\leq\mt$ element of $\R_+^2$ and every continuous bounded function $\Phi:\,\R^{n+1}\to\R$,
$$
\E[\Phi(X_{\ms_1},\cdots,X_{\ms_n},X_\ms)(X_\mt-X_\ms)]\geq0.
$$
\end{rem}

\subsection{The Cox-Hobson algorithm and the mean residual life ordering}
A connection between the MRL ordering and the submartingale property is deduced from the Cox-Hobson algorithm which extends that of Azéma and Yor (see \citep{AY79}). Let $\left(\mu_\mt,\mt\in\R_+^2\right)$ be a family of integrable probability measures such that, for every $\mt\in\R_+^2$, $m_{\mu_\mt}:=\int_\R y\mu_\mt(dy)>0$, and let $(B_v,v\geq0)$ be a standard Brownian motion started at $0$. The Cox-Hobson algorithm provides a family of stopping times $\left(T^{}_{\mu_\mt},\mt\in\R_+^2\right)$ such that, for every $\mt\in\R_+^2$,
\begin{enumerate}
\item[C1.]$\left(B^+_{T^{}_{\mu_\mt}\wedge v},v\geq0\right)$ is uniformly integrable,
\item[C2.]$B_{T^{}_{\mu_\mt}}$ has law $\mu_\mt$.
\end{enumerate}
Precisely, for a fixed $\mt$, consider the convex function $\pi_{\mu_\mt}$ given by:
$$
\forall\,x\in\R,\quad\pi_{\mu_\mt}(x)=\int_\R|y-x|\mu_\mt(dy)+m_{\mu_\mt}.
$$
For every $\theta\in[-1,1]$, define
$$
u_{\mu_\mt}(\theta)=\inf\{y\in\R:\,\pi_{\mu_\mt}(y)+\theta(x-y)\leq\pi_{\mu_\mt}(x),\text{ for every }x\in\R\}.
$$
and
$$
\forall\,\theta\in[-1,1],\quad z_{\mu_\mt}(\theta)=\frac{\pi_{\mu_\mt}(u_{\mu_\mt}(\theta))-\theta u_{\mu_\mt}(\theta)}{1-\theta}.
$$
Define also the function $b_{\mu_\mt}$ by:
$$
\forall\,\alpha\in\R_+,\quad b_{\mu_\mt}(\alpha)=u_{\mu_\mt}\left(z^{-1}_{\mu_\mt}(\alpha)\right),
$$
where
$$
z^{-1}_{\mu_\mt}(\alpha)=\inf\{\theta\in[-1,1]:\,z_{\mu_\mt}(\theta)\geq\alpha\}.
$$
\begin{figure}


\ifx\JPicScale\undefined\def\JPicScale{.5}\fi
\psset{xunit=.4cm,yunit=.4cm}
\psset{dotsize=0.7 2.5,dotscale=1 1,fillcolor=black}
\begin{pspicture}(-16,-2)(16,15)
\psline[linewidth=1pt,arrowinset=0]{->}(-15,0)(15,0)
\psline[linewidth=1pt,arrowinset=0]{->}(0,-1)(0,14)
\psline[linewidth=0.8pt](0,0)(14,12)
\psline[linewidth=0.8pt](0,0)(-14,12)
\psline[linewidth=0.6pt, linestyle=dashed](2,0)(-12,12)
\psline[linewidth=1pt](-11.5,11.87) (-10,10.56) 
\pscurve[linewidth=1pt](-10,10.56)(0,4.7) (12,10.6)  
\psline[linewidth=1pt](12,10.6) (13.3,11.65)
\psline[linewidth=0.6pt](-6,6.31)(3.85,3.3)
\psline[linewidth=0.6pt, linestyle=dashed](3.85,0)(3.85,3.3)
\psline[linewidth=0.6pt, linestyle=dashed](-1.5,0)(-1.5,5)
\uput[dr](3.85,0){$z$}
\uput[d](-1.5,0){$b_{\mu_\mt}(z)$}
\uput[d](2,0){$2m_{\mu_\mt}$}
\uput[d](15,0){$x$}
\uput[dl](0,14){$\pi_{\mu_\mt}(x)$}
\end{pspicture}
\caption{$\pi_{\mu_\mt}$ for a $\mu_\mt$ with positive mean and such that $\text{supp}(\mu_\mt)=\R$.}
\label{pict:MTP2MRL}
\end{figure}
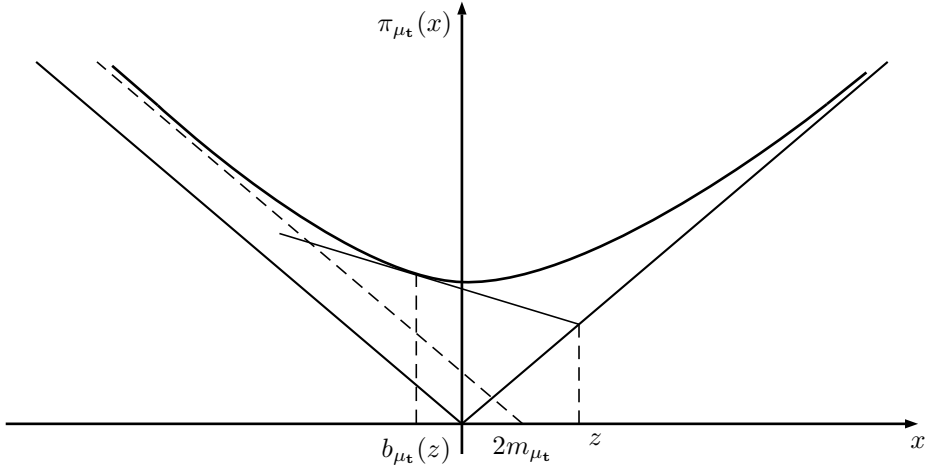
We refer to Figure \ref{pict:MTP2MRL} for a pictorial representation of the function $\pi_{\mu_\mt}$ and of the image $b_{\mu_\mt}(z)$ of a point $z\in\R_+$ under $b_{\mu_\mt}$. Let $T_{\mu_\mt}$ be the stopping time given by:
$$
T_{\mu_\mt}=\inf\{v\geq0:\,b_{\mu_\mt}(S_v)\geq B_v\}
=\inf\{v\geq0:\,S_v\geq b^{-1}_{\mu_\mt}(B_v)\},
$$
where $S_v=\sup_{w\leq v}B_w$ and where, for every $y\in\R$, $b_{\mu_\mt}^{-1}(y)=\inf\{z\in\R_+:\,b_{\mu_\mt}(z)\geq y\}$.
Cox and Hobson proved that $T_{\mu_\mt}$ satisfies Conditions C1 and C2 (see \citep[Theorem 12]{CH06}). They also established that Condition C1 is equivalent to the assertion that $T_{\mu_\mt}$ is minimal for $(B_v,v\geq0)$ in the sense that if there is a stopping time $R\leq T_{\mu_\mt}$ which embeds $\mu_\mt$, i.e. such that $B_R$ has law $\mu_\mt$, then $R=T_{\mu_\mt}$ a.s. Moreover, by Theorem 12 in \citep{CH06}, $T_{\mu_\mt}$ is optimal in the sense that $T_{\mu_\mt}$ maximizes $\Pb(S_R\geq x)$ amongst all minimal stopping times $R$ embedding $\mu_\mt$, uniformly in $x$. We mention that Beiglb{\"o}ck, Cox and Huesmann \citep{BCH17} develop recently a transport-based approach to the Skorokhod embedding problem which allows them to derive all known and a variety of new optimal solutions.
\noindent
One may observe that $T_{\mu_\mt}$ rewrites in terms of $\Psi_{\mu_\mt}$. Indeed, for every $x\in\R$,
$$
u^{-1}_{\mu_\mt}(x)=\inf\{\theta\in[-1,1]:\,u_{\mu_\mt}(\theta)\geq x\}=\partial\pi_{\mu_\mt}(x)=1-2\mu_\mt([x,+\infty[),
$$
where $\partial\pi_{\mu_\mt}$ denotes the left-derivative of $\pi_{\mu_\mt}$, and
$$
\pi_{\mu_\mt}\left(u_{\mu_\mt}\left(u^{-1}_{\mu_\mt}(x)\right)\right)-u^{-1}_{\mu_\mt}(x)u_{\mu_\mt}\left(u^{-1}_{\mu_\mt}(x)\right)=\pi_{\mu_\mt}(x)-xu^{-1}_{\mu_\mt}(x).
$$
Then, as in \citep[Section 3.2]{Cox04}, one may prove that
$$
b^{-1}_{\mu_\mt}(x)=z_{\mu_\mt}\left(u^{-1}_{\mu_\mt}(x)\right)=\Psi_{\mu_\mt}(x).
$$
Note that C1 is also equivalent to $\E\left[B_{T_{\mu_\mt}}|\mathcal{F}_R\right]\geq B_R$ for all stopping times $R\leq T_{\mu_\mt}$. In particular, if the family $\left(T_{\mu_\mt},\mt\in\R_+^2\right)$ is non-decreasing a.s., then, for every $\ms\leq\mt$, $\E\left[B_{T_{\mu_\mt}}|\mathcal{F}_{T_{\mu_\ms}}\right]\geq B_{T_{\mu_\ms}}$ which means that $\left(B_{T_{\mu_\mt}},\mt\in\R_+^2\right)$ is a submartingale, and, according to C2, this submartingale has marginals $\left(\mu_\mt,\mt\in\R_+^2\right)$. Moreover, if, for every $\mt\in\R_+^2$, $\mu_\mt$ has density, then, using Madan-Yor's argument, one may show that $\left(B_{T_{\mu_\mt}},\mt\in\R_+^2\right)$ is an incomplete Markov process in the sense that, for every $\ms_0<\ms_1<\cdots<\ms_n<\mt$ ($n\in\N^\ast$) elements of $\R_+^2$, every real numbers $x_0,x_1,\cdots,x_n$ and every Borel subset $A$ of $\R$,
$$
\Pb\left(\left.B_{T_{\mu_\mt}}\in A\right|B_{T_{\mu_{\ms_0}}}=x_0,B_{T_{\mu_{\ms_1}}}=x_1,\cdots, B_{T_{\mu_{\ms_n}}}=x_n\right) 
= \Pb\left(\left.B_{T_{\mu_\mt}}\in A\right|B_{T_{\mu_{\ms_n}}}=x_n\right).
$$
We refer to \citep{Ra83} (see also \citep{JLu92}) for the definition and some properties of incomplete Markov processes.
In the case where there is a measure $\mu_\ms$ with atoms, the epigraph of $\Psi_{\mu_\ms}$ admits vertical slopes, and then, conditionally on $\mathcal{F}_{T_{\mu_\ms}}$, any future time $T_{\mu_\mt}$, where $\ms\leq\mt$, depend on both $B_{T_{\mu_\ms}}$ and $S_{T_{\mu_\ms}}$ which implies that $\left(B_{T_{\mu_\mt}},\mt\in\R_+^2\right)$ does not enjoy necessarily the incomplete Markov property.
On the other hand, since $T_{\mu_\mt}=\inf\{v\geq0:\,S_v\geq \Psi_{\mu_\mt}(B_v)\}$ is the first time the process $(B_v,S_v;v\geq0)$ hits the epigraph of $\Psi_{\mu_\mt}$, then the family $\left(T_{\mu_\mt},\mt\in\R_+^2\right)$ is non-decreasing a.s. if, and only if, for every $\ms\leq\mt$, the epigraph of $\Psi_{\mu_\mt}$ includes that of $\Psi_{\mu_\ms}$ which means that $\left(\mu_\mt,\mt\in\R_+^2\right)$ is a MRL process.
\begin{rem}\label{rem:MRLTwoDSubMart}
Even if there exists some $\mu_{\mt_0}$ with negative mean, the MRL ordering of $\left(\mu_\mt,\mt\in\R_+^2\right)$ is sufficient for the construction of an associated submartingale to $\left(\mu_\mt,\mt\in\R_+^2\right)$ using the Cox-Hobson algorithm. Indeed, if $m_0$ is a real number satisfying $m_0<\int_{\R}y\mu_{(0,0)}(dy)$, and if $g_{\#}\mu_\mt$ denotes the image of $\mu_\mt$ under $g:\,y\longmapsto y-m_0$, then
$\left(g_{\#}\mu_\mt,\mt\in\R_+^2\right)$ is still a MRL process, 
and, for every $\mt\in\R_+^2$, $\int_\R z\,g_{\#}\mu_\mt(dz)>0$. 
Precisely, for every $x\in\R$,
$$
\mt\longmapsto \Psi_{g_{\#}\mu_\mt}(x)=\Psi_{\mu_\mt}(x+m_0)-m_0\text{ is non-decreasing,}
$$
and, as $\mt\longmapsto\int_{\R}y\mu_\mt(dy)$ is non-decreasing (see Point 1 of Remark \ref{rem:MRLPsiRLmu}),
$$
\int_\R z\,g_{\#}\mu_\mt(dz)=\int_\R(y-m_0)\mu_\mt(dy)>\int_\R y\mu_\mt(dy)-\int_{\R}y\mu_{(0,0)}(dy)\geq0.
$$
Now, if $\left(B_v,v\geq0\right)$ is a standard Brownian motion issued from $0$, and if we denote by $T_{g_{\#}\mu_\mt}$ the Cox-Hobson stopping which embeds $g_{\#}\mu_\mt$, then $\left(B_{T_{g_{\#}\mu_\mt}}+m_0,\mt\in\R_+^2\right)$ is a submartingale associated to $\left(\mu_\mt,\mt\in\R_+^2\right)$.
\end{rem}
It follows from the preceding remark that MRL ordering is a sufficient condition for the existence of an associated submartingale to a given integrable two-parameter process. 

\subsection{Organization of the paper}
In Section 2, we recall the definition of MTP$_2$ function and discuss some multivariate total positivity properties of two-parameter MRL processes. We prove that the integrated survival function of an integrable two-parameter process is MTP$_2$ if, and only if it is TP$_2$ in each pair of variables when the remaining variable is fixed. In Section 3, we provide several examples of two-parameter MRL processes. In particular, we show that the MTP$_2$ property of certain two-parameter MRL processes is useful to generate other MRL processes with the same property. The last section is devoted to explicit construction of associated submartingales to certain non-MRL processes.

\section{Total positivity properties of MRL processes}
\subsection{Multivariate totally positive functions}\label{ssec:MTP2}
\begin{defi}
Let $n$ be an integer such that $n\geq2$ and let $\mathcal{I}_1,\mathcal{I}_2,\cdots,\mathcal{I}_n$ be $n$ totally ordered sets. A nonnegative function $f$ over $\mathcal{I}=\prod\limits_{i=1}^n\mathcal{I}_i$ is said to be multivariate totally positive of order 2 (MTP$_2$) if, for every $(\mathbf{x},\mathbf{y})\in\mathcal{I}\times\mathcal{I}$,
\begin{equation}\label{eq:MTP2}
f(\mathbf{x}\wedge\mathbf{y})f(\mathbf{x}\vee\mathbf{y})\geq f(\mathbf{x})f(\mathbf{y}),
\end{equation}
where $\mathbf{x}=(x_1,x_2,\cdots,x_n)$, $\mathbf{y}=(y_1,y_2,\cdots,y_n)$,
$$
\mathbf{x}\wedge\mathbf{y}=(x_1\wedge y_1,x_2\wedge y_2,\cdots,x_n\wedge y_n)=\left(\min\{x_1,y_1\},\min\{x_2,y_2\},\cdots,\min\{x_n,y_n\}\right)
$$
and
$$
\mathbf{x}\vee\mathbf{y}=(x_1\vee y_1,x_2\vee y_2,\cdots,x_n\vee y_n)=\left(\max\{x_1,y_1\},\max\{x_2,y_2\},\cdots,\max\{x_n,y_n\}\right).
$$
\end{defi}
\begin{rem}
\item[1.]When $n=2$, a nonnegative function $f$ satisfying (\ref{eq:MTP2}) is said to be {\it totally positive of order 2} (TP$_2$). We refer to \citep{Ka68} for examples and properties of totally positive functions.
\item[2.]Suppose that $n\geq3$. By definition, a MTP$_2$ function is TP$_2$ in each pair of variables when the remaining variables are kept constant. But the converse is false. When $n=3$, Kemperman \citep[Page 330]{Kem77} provided an example of function which is TP$_2$ in each pair of variables when the remaining variable is fixed, and which is not MTP$_2$. Karlin and Rinot \citep{KaR80} proved that the converse holds when we restrict ourselves to functions that have interval support. We prove that the result of Karlin and Rinot extends to the integrated survival functions of integrable two-parameter processes. We point out that these functions do not have necessarily interval support.
\end{rem}
Many examples of MTP$_2$ density functions have been exhibited by Karlin and Rinot \citep{KaR80}. These authors also provided several transformations that preserve the MTP$_2$ property. In particular, they proved the following composition formula which allows to generate other MTP$_2$ functions.
\begin{prop}\label{prop:CompoMTP2}\citep[Proposition 3.4]{KaR80}.
Let $n$, $m$, $l$ be positive integers. Let $\mathcal{I}=\prod\limits_{i=1}^n\mathcal{I}_i$, $\mathcal{J}=\prod\limits_{i=1}^m\mathcal{J}_i$, $\mathcal{K}=\prod\limits_{i=1}^l\mathcal{K}_i$, where  $\mathcal{I}_i$, $\mathcal{J}_i$ and $\mathcal{K}_i$ are totally ordered sets. Let $f$ be MTP$_2$ on $\mathcal{I}\times\mathcal{J}$ and $g$ be MTP$_2$ on $\mathcal{J}\times\mathcal{K}$. Define
$$
h(\mathbf{x},\mathbf{z})=\int_{\mathcal{J}}f(\mathbf{x},\mathbf{y})g(\mathbf{y},\mathbf{z})\varrho(d\mathbf{y}),
$$
where $\varrho=\varrho_1\times\cdots\times\varrho_m$ and $\varrho_i$ is a $\sigma$-finite positive measure on $\mathcal{J}_i$. Then $h$ is MTP$_2$ on $\mathcal{I}\times\mathcal{K}$.
\end{prop}
In Section 3, we shall apply a special case of the above result to exhibit several MRL processes which have MTP$_2$ integrated survival functions.

\subsection{Integrated survival functions of MRL processes}
We give some total positivity properties of the integrated survival functions of MRL processes. Let $\mu=\left(\mu_\mt,\mt\in\R_+^2\right)$ be an integrable process, and let $C_\mu$ be its integrable survival function, defined by:
$$
\forall\,(\mt,x)\in\R_+^2\times\R,\quad C_\mu(\mt,x)=\int_{[x,+\infty[}(y-x)\mu_\mt(dy).
$$
\begin{rem}\label{rem:CMuSuppCond}
Let $r_{\mu_\mt}$ be the upper bound of the support of $\mu_\mt$ (see Point 2 of Remark \ref{rem:MRLPsiRLmu}). Then $C_\mu(\mt,x)=0$ if, and only if $r_{\mu_\mt}\leq x$.
\end{rem}
We recall that an integrable probability measure is entirely determined by its integrated survival function. This is the purpose of the next result, borrowed from \citep[Section 2]{HR12} (see also \citep[Theorem 1.5.10]{MS02}).
\begin{prop}
Let $\nu$ be an integrable probability measure and let $C_\nu$ denote the integrated survival function of $\nu$, i.e. the function defined on $\R$ by $C_\nu(x)=\int_{[x,+\infty[}(y-x)\mu_\mt(dy)$. Then $C_\nu$ enjoys the following properties:
\begin{enumerate}
\item[i)]$C_\nu$ is a convex, nonnegative function on $\R$,
\item[ii)]$\lim\limits_{x\to+\infty}C_\nu(x)=0$,
\item[iii)]there exists $l\in\R$ such that $\lim\limits_{x\to-\infty}C_\nu(x)+x=l$.
\end{enumerate}
Conversely, if a function $C$ satisfies the above three properties, then there exists a unique integrable probability measure $\nu$ such that $C_\nu=C$, i.e. $C$ is the integrated survival function of $\nu$. Precisely, $\nu$ is the second order derivative of $C$ in the sense of distributions, and $l=\int_\R y\nu(dy)$. 
\end{prop}
We have the same characterization of MRL ordering in terms of integrated survival functions as in the one-parameter case and the proof follows the same lines as that of Theorem 3.3 in \citep{Bo15}.
\begin{theorem}\label{theo:MRLTPP2}
The family $\left(\mu_\mt,\mt\in\R_+^2\right)$ is non-decreasing in the MRL ordering if, and only if its integrated survival function $C_\mu$ satisfies:
\begin{equation}\label{eq:MRLTPP2}
\forall\,\mt_1\leq\mt_2,\,\forall\,x_1\leq x_2,\quad \det\left(
\begin{array}{cc}
C_\mu(\mt_1,x_1)&C_\mu(\mt_1,x_2)\\
C_\mu(\mt_2,x_1)&C_\mu(\mt_2,x_2)
\end{array}
\right)\geq0.
\end{equation}
\end{theorem}
\begin{rem}\label{rem:KemMRLTPP2}
Inequality (\ref{eq:MRLTPP2}) is equivalent to the following assertion:
\begin{equation}\label{eq:MTPP2} 
C_\mu\text{ is TP}_2\text{ in }(t,x)\text{ when }\tpr\text{ is fixed, and TP}_2\text{ in }(\tpr,x)\text{ when }t\text{ is fixed.}
\end{equation}
Such equivalence does not hold in general. Precisely, (\ref{eq:MRLTPP2}) strictly implies (\ref{eq:MTPP2}). Indeed, consider the following example borrowed from \citep[Page 330, Proof of (ii)]{Kem77}: Let $u$ and $v$ be two positive real numbers, and let $\phi:\,\R_+^3\to\R_+$ be defined by:
\begin{equation}\label{eq:CountexaKem}
\phi(t,\tpr,x)=\left\{
\begin{array}{cl}
u&\text{if }(t,\tpr,x)\in[0,1]\times[0,1]\times]1,2],\\
v&\text{if }(t,\tpr,x)\in]1,2]\times]1,2]\times[0,1],\\
0&\text{otherwise.}
\end{array}
\right.
\end{equation}
Then $\phi$ is TP$_2$ in $(t,x)$ when $\tpr$ is fixed, and TP$_2$ in $(\tpr,x)$ when $t$ is fixed (see e.g. \citep[Page 330, Proof of (ii)]{Kem77}). But, for every $(t_1,\tpr_1,x_1)\in[0,1]^3$ and $(t_2,\tpr_2,x_2)\in]1,2]^3$, 
$$
\phi(t_1,\tpr_1,x_1)\phi(t_2,\tpr_2,x_2)=0<uv=\phi(t_1,\tpr_1,x_2)\phi(t_2,\tpr_2,x_1)
$$
which shows that $\phi$ does not satisfies (\ref{eq:MRLTPP2}). To show that (\ref{eq:MTPP2}) implies (\ref{eq:MRLTPP2}) in the case of integrated survival functions, we apply a characterization of MRL ordering in the one-parameter case (see \citep[Theorem 3.3]{Bo15}). Let $\left(\mu_\mt,\mt\in\R_+^2\right)$ be an integrable process whose integrated survival function satisfies (\ref{eq:MTPP2}). Let $\mt_1=(t_1,\tpr_1)$ and $\mt_2=(t_2,\tpr_2)$ be such that $\mt_1\leq\mt_2$. Since $C_\mu$ is TP$_2$ in $(t,x)$ when $\tpr$ is fixed, we have $\mu_{\mt_1}=\mu_{(t_1,\tpr_1)}\leq_{mrl}\mu_{(t_2,\tpr_1)}$, i.e. $\mu_{(t_2,\tpr_1)}$ dominates $\mu_{\mt_1}$ in the MRL ordering, and, since $C_\mu$ is TP$_2$ in $(\tpr,x)$ when $t$ is fixed, we also have $\mu_{(t_2,\tpr_1)}\leq_{mrl}\mu_{(t_2,\tpr_2)}=\mu_{\mt_2}$. Then, as the MRL ordering is transitive, $\mu_{\mt_1}\leq_{mrl}\mu_{\mt_2}$ which means that $C_\mu$ satisfies (\ref{eq:MRLTPP2}).
\end{rem}

One may observe that the MRL ordering is not preserved by some convex transformations. For instance, if $\mu=\left(\mu_\mt,\mt\in\{1,2\}\times\R_+\right)$ is a MRL process, and if $\nu=\left(\nu_\mt,\mt\in[0,1]\times\R_+\right)$ is the process given by
$$
\forall\,\mt\in\R^2_+,\quad\nu_\mt=(1-t)\mu_{(1,\tpr)}+t\mu_{(2,\tpr)},
$$
then $\nu$ is not necessarily a MRL process. Indeed, the integrated survival function $C_\nu$ of $\nu$, defined on $[0,1]\times\R_+\times\R$ by
$$
C_\nu(\mt,x)=(1-t)C_\mu(1,\tpr,x)+tC_\mu(2,\tpr,x),
$$
where $C_\mu$ denotes the integrated survival function of $\mu$, is not always TP$_2$ in $(\tpr,x)$ when $t$ is fixed. We introduce a class of MRL processes which is closed by several convex transformations, and from which are generated other MRL processes that belong to the same class.
\begin{theorem}\label{theo:MRLMTP2Tail}
Let $\left(\mu_\mt,\mt\in\R_+^2\right)$ be an integrable process, and let $C_\mu$ denote its integrated survival function. If $C_\mu$ is TP$_2$ in every pair of variables when the remaining variable is fixed, then $C_\mu$ is MTP$_2$.
\end{theorem}
\begin{rem}
Theorem \ref{theo:MRLMTP2Tail} is not valid for all nonnegative functions. We know from \citep{Lo53} (see also \citep{Kem77} and \citep{KaR80}) that such a result holds for positive functions. Karlin and Rinot \citep[Proposition 2.1]{KaR80} proved that Theorem \ref{theo:MRLMTP2Tail} still applies to nonnegative functions which have interval support (see also \citep[Proposition 3.5]{Fal17} ): as in \citep[Page 7]{Fal17}, we say that $\phi:\,\R_+^2\times\R\to\R_+$ {\it has interval support} if, for every $(\mt_1,x_1)$ and $(\mt_2,x_2)$ in $\R_+^2\times\R$, $\phi(\mt_1,x_1)\phi(\mt_2,x_2)>0$ implies $\phi(\mt,x)>0$ for every $(\mt,x)\in\R_+^2\times\R$ satisfying $(\mt_1\wedge\mt_2,x_1\wedge x_2)\leq (\mt,x)\leq(\mt_1\vee\mt_2,x_1\vee x_2)$. But as it is shown in Propositions \ref{prop:TheoMRLMTP2} and \ref{prop:ExaMTP2ModifNu} below, there are several integrated survival functions which do not have interval support.
\end{rem}
\begin{proof}[Proof of Theorem \ref{theo:MRLMTP2Tail}.]
Let $C_\mu$ be the integrated survival function of an integrable process $\left(\mu_\mt,\mt\in\R_+^2\right)$ which is TP$_2$ in each pair of variables when the remaining variable is fixed. For every $\mt\in\R_+^2$, we denote by $r_{\mu_\mt}$ the upper bound of the support of $\mu_\mt$. Let $\mt_1$ and $\mt_2$ be elements of $\R_+^2$, and let $x_1$ and $x_2$ be two real numbers. We suppose without loss of generality that $x_1\leq x_2$. We wish to prove the following inequality.
\begin{equation}\label{eq:ProofMRLMTP2a}
C_\mu(\mt_1\wedge\mt_2,x_1)C_\mu(\mt_1\vee\mt_2,x_2)\geq C_\mu(\mt_1,x_1)C_\mu(\mt_2,x_2).
\end{equation}
We first show that the right-hand side of (\ref{eq:ProofMRLMTP2a}) vanishes when the left-hand side equals zero. The left-hand side of (\ref{eq:ProofMRLMTP2a}) equals zero if, and only if at least one of $C_\mu(\mt_1\wedge\mt_2,x_1)$ and $C_\mu(\mt_1\vee\mt_2,x_2)$ equals zero.  
\begin{itemize}
\item If $C_\mu(\mt_1\vee\mt_2,x_2)=0$, then, by Remark \ref{rem:CMuSuppCond}, $r_{\mu_{\mt_1\vee\mt_2}}\leq x_2$. Moreover, since $C_\mu$ is TP$_2$ in $(t,x)$ when $\tpr$ is fixed, and TP$_2$ in $(\tpr,x)$ when $t$ is fixed, it follows from Remark \ref{rem:KemMRLTPP2} that $C_\mu$ satisfies Condition (\ref{eq:MRLTPP2}) in Theorem \ref{theo:MRLTPP2}. Hence $\left(\mu_\mt,\mt\in\R_+^2\right)$ is a MRL process, and we deduce from Point 2 of Remark \ref{rem:MRLPsiRLmu} that $\mt\longmapsto r_{\mu_\mt}$ is non-decreasing. Then $r_{\mu_{\mt_2}}\leq r_{\mu_{\mt_1\vee\mt_2}}\leq x_2$ which implies that $C_\mu(\mt_2,x_2)=0$ and, as a consequence, that the right-hand side of (\ref{eq:ProofMRLMTP2a}) equals zero.
\item Suppose that $C_\mu(\mt_1\wedge\mt_2,x_1)=0$. Because $C_\mu$ is TP$_2$ in $(t,\tpr)$ when $x$ is fixed,
\begin{equation}\label{eq:ProofMRLMTP2b}
C_\mu(\mt_1\wedge\mt_2,x_1)C_\mu(\mt_1\vee\mt_2,x_1)\geq C_\mu(\mt_1,x_1)C_\mu(\mt_2,x_1).
\end{equation} 
Since $C_\mu(\mt_1\wedge\mt_2,x_1)=0$, the right-hand side of (\ref{eq:ProofMRLMTP2b}) equals zero which implies that at least one of $C_\mu(\mt_1,x_1)$ and $C_\mu(\mt_2,x_1)$ equals zero. If $C_\mu(\mt_1,x_1)=0$, then the right-hand side of (\ref{eq:ProofMRLMTP2a}) equals zero. If $C_\mu(\mt_2,x_1)=0$, then $r_{\mu_{\mt_2}}\leq x_1\leq x_2$ which implies that $C_\mu(\mt_2,x_2)=0$. Thus, the right-hand side of (\ref{eq:ProofMRLMTP2a}) equals zero.
\end{itemize} 
Now, suppose that the left-hand side of (\ref{eq:ProofMRLMTP2a}) is positive. In particular $C_\mu(\mt_1\wedge\mt_2,x_1)$ is positive and, as a consequence $x_1<r_{\mu_{\mt_1\wedge\mt_2}}$. Since $\mt\longmapsto r_{\mu_\mt}$ is non-decreasing, $x_1< r_{\mu_{\mt_1\wedge\mt_2}}\leq r_{\mu_{\mt_2}}$ which implies that $C_\mu(\mt_2,x_1)$ is positive too. Hence, we may write:
\begin{equation}\label{eq:MRLMTP2Proofa}
C_\mu(\mt_1\wedge\mt_2,x_1)C_\mu(\mt_1\vee\mt_2,x_2)=\frac{C_\mu(\mt_1\wedge\mt_2,x_1)}{C_\mu(\mt_2,x_1)}[C_\mu(\mt_2,x_1)C_\mu(\mt_1\vee\mt_2,x_2)].
\end{equation}
Because $C_\mu$ is TP$_2$ in $(t,x)$ when $\tpr$ is fixed, and TP$_2$ in $(\tpr,x)$ when $t$ is fixed, we deduce from Remark \ref{rem:KemMRLTPP2} and from (\ref{eq:MRLTPP2}) that
\begin{equation}\label{eq:MRLMTP2Proofb}
C_\mu(\mt_2,x_1)C_\mu(\mt_1\vee\mt_2,x_2)\geq C_\mu(\mt_2,x_2)C_\mu(\mt_1\vee\mt_2,x_1).
\end{equation}
Combining (\ref{eq:MRLMTP2Proofa}) and (\ref{eq:MRLMTP2Proofb}), we obtain
\begin{equation}\label{eq:MRLMTP2Proofc}
C_\mu(\mt_1\wedge\mt_2,x_1)C_\mu(\mt_1\vee\mt_2,x_2)\geq\frac{C_\mu(\mt_2,x_2)}{C_\mu(\mt_2,x_1)}[C_\mu(\mt_1\vee\mt_2,x_1)C_\mu(\mt_1\wedge\mt_2,x_1)].
\end{equation}
Moreover, since $C_\mu$ is TP$_2$ in $(t,\tpr)$ when $x$ is constant,
\begin{equation}\label{eq:MRLMTP2Proofd}
C_\mu(\mt_1\wedge\mt_2,x_1)C_\mu(\mt_1\vee\mt_2,x_1)\geq C_\mu(\mt_1,x_1)C_\mu(\mt_2,x_1).
\end{equation}
Then (\ref{eq:MRLMTP2Proofc}) and (\ref{eq:MRLMTP2Proofd}) yield
$$
C_\mu(\mt_1\wedge \mt_2,x_1)C_\mu(\mt_1\vee\mt_2,x_2)\geq C_\mu(\mt_1,x_1)C_\mu(\mt_2,x_2)
$$
which completes the proof. 
\end{proof}

In the next section, we exhibit many examples of MRL processes among which there are processes that possess MTP$_2$ integrated survival functions.

\section{Some examples of two-parameter MRL processes}
We provide several examples of MRL processes among which there are processes with MTP$_2$ integrated survival functions. In particular, the MTP$_2$ property of these processes is useful to generate other processes having MTP$_2$ integrated survival functions.

\subsection{A family of diatomic MRL processes}
Here is an example of a family of diatomic processes to which Theorem \ref{theo:MRLMTP2Tail} applies. 

\begin{prop}\label{prop:TheoMRLMTP2}
Let $\ve\in(0,1)$ and $r\in\R$. Let $\left(\mu_\mt^{\ve},\mt\in\R_+^2\right)$ be the process given by: for every $t\in\R_+$, $\mu_{(t,0)}^\ve=\mu_{(0,t)}^\ve=\delta_r$, and, for every $\mt=(t,\tpr)\in\R_+^\ast\times\R_+^\ast$,
$$
\mu_\mt^\ve=\frac{\tpr}{t+\tpr}\delta_{r-(1-\ve)t}+\frac{t}{t+\tpr}\delta_{r+\tpr},
$$
where $\R_+^\ast$ denotes the set of positive numbers, and, for every $a\in \R$, $\delta_a$ denotes the Dirac measure at point $a$. Then the integrated survival function $C_{\mu^\ve}$ of $\left(\mu^\ve_\mt,\mt\in\R_+^2\right)$ is TP$_2$ in each pair of variables when the remaining variable is held constant. 
\end{prop}
\begin{proof}
We first show that $C_{\mu^\ve}$ is TP$_2$ in $(t,x)$ when $\tpr$ is fixed, and TP$_2$ in $(\tpr,x)$ when $t$ is fixed. By Remark \ref{rem:KemMRLTPP2}, it suffices to prove that $\left(\mu^\ve_\mt,\mt\in\R_+^2\right)$ is a MRL process. For every $\mt\in\left(\R_+^\ast\right)^2=\R_+^\ast\times\R_+^\ast$, the Hardy-Littlewood function $\Psi_{\mu_\mt^\ve}$ of $\mu_\mt^\ve$ is defined by:
$$
\forall\,x\in\R,\quad\Psi_{\mu_\mt^\ve}(x)=\left\{
\begin{array}{cl}
r+\ve\dfrac{t\tpr}{t+\tpr}&\text{if }x\leq r-(1-\ve)t,\\
r+\tpr&\text{if }r-(1-\ve)t<x\leq r+\tpr,\\
x&\text{if }x>r+\tpr.
\end{array}
\right.
$$
Since $(t,\tpr)\longmapsto\dfrac{t\tpr}{t+\tpr}$ is non-decreasing on $\left(\R_+^\ast\right)^2$, it is not difficult to show that, for every $x\in\R$ and every $\mt_1\leq\mt_2$, $\Psi_{\mu_{\mt_1}^\ve}(x)\leq \Psi_{\mu_{\mt_2}^\ve}(x)$. It remains to prove that $C_{\mu^\ve}$ is TP$_2$ in $(t,\tpr)$ when $x$ is fixed. The function $C_{\mu^\ve}$ is given by:
$$
\forall\,(\mt,x)\in\left(\R_+^\ast\right)^2\times\R,\quad C_{\mu^\ve}(\mt,x)=\left\{
\begin{array}{cl}
r-x+\ve\dfrac{t\tpr}{t+\tpr}&\text{if }x<r-(1-\ve)t,\\
\dfrac{t(\tpr+r-x)}{t+\tpr}&\text{if }r-(1-\ve)t\leq x<r+\tpr,\\
0&\text{if }x\geq r+\tpr.
\end{array}
\right.
$$
Let $t_1\leq t_2$ and $\tpr_1\leq\tpr_2$ be positive real numbers. We wish to prove that, for every $x\in\R$,
\begin{equation}\label{eq:exaMRLMTP2}
C_{\mu^\ve}(t_1,\tpr_1,x)C_{\mu^\ve}(t_2,\tpr_2,x)-C_{\mu^\ve}(t_1,\tpr_2,x)C_{\mu^\ve}(t_2,\tpr_1,x)\geq0.
\end{equation}
We may write 
$$
\R=]-\infty,r-(1-\ve)t_2[\cup[r-(1-\ve)t_2,r-(1-\ve)t_1[ \cup[r-(1-\ve)t_1,r+\tpr_1[\cup[r+\tpr_1,+\infty[.
$$
If $x\in]-\infty,r-(1-\ve)t_2[$, then
\begin{align*}
&C_{\mu^\ve}(t_1,\tpr_1,x)C_{\mu^\ve}(t_2,\tpr_2,x)-C_{\mu^\ve}(t_1,\tpr_2,x)C_{\mu^\ve}(t_2,\tpr_1,x)\\
&=\left(r-x+\ve\dfrac{t_1\tpr_1}{t_1+\tpr_1}\right)\left(r-x+\ve\dfrac{t_2\tpr_2}{t_2+\tpr_2}\right)-\left(r-x+\ve\dfrac{t_1\tpr_2}{t_1+\tpr_2}\right)\left(r-x+\ve\dfrac{t_2\tpr_1}{t_2+\tpr_1}\right)\\
&=\ve(r-x)\left(\dfrac{t_1\tpr_1}{t_1+\tpr_1}+\dfrac{t_2\tpr_2}{t_2+\tpr_2}-\dfrac{t_1\tpr_2}{t_1+\tpr_2}-\dfrac{t_2\tpr_1}{t_2+\tpr_1}\right)+\\
&\qquad\ve^2t_1\tpr_1t_2\tpr_2\left(\dfrac{1}{(t_1+\tpr_1)(t_2+\tpr_2)}-\dfrac{1}{(t_1+\tpr_2)(t_2+\tpr_1)}\right).
\end{align*}
Since, on $\left(\R_+^\ast\right)^2$, $(t,\tpr)\longmapsto\dfrac{t\tpr}{t+\tpr}$ is supermodular and $(t,\tpr)\longmapsto\dfrac{1}{t+\tpr}$ is TP$_2$, we have
$$
\dfrac{t_1\tpr_1}{t_1+\tpr_1}+\dfrac{t_2\tpr_2}{t_2+\tpr_2}-\dfrac{t_1\tpr_2}{t_1+\tpr_2}-\dfrac{t_2\tpr_1}{t_2+\tpr_1}\geq0
$$ 
and
$$
\dfrac{1}{(t_1+\tpr_1)(t_2+\tpr_2)}-\dfrac{1}{(t_1+\tpr_2)(t_2+\tpr_1)}\geq0
$$
respectively which implies that (\ref{eq:exaMRLMTP2}) holds.\\
Suppose that $x\in[r-(1-\ve)t_2,r-(1-\ve)t_1[$.  Then
\begin{align*}
&C_{\mu^\ve}(t_2,\tpr_2,x)C_{\mu^\ve}(t_1,\tpr_1,x)-C_{\mu^\ve}(t_2,\tpr_1,x)C_{\mu^\ve}(t_1,\tpr_2,x)\\
&=\dfrac{t_2(\tpr_2+r-x)}{t_2+\tpr_2}\left(r-x+\ve\dfrac{t_1\tpr_1}{t_1+\tpr_1}\right)-\dfrac{t_2(\tpr_1+r-x)}{t_2+\tpr_1}\left(r-x+\ve\dfrac{t_1\tpr_2}{t_1+\tpr_2}\right)\\
&=t_2(r-x)\left(\dfrac{\tpr_2+r-x}{t_2+\tpr_2}-\dfrac{\tpr_1+r-x}{t_2+\tpr_1}\right)+\ve t_1t_2\left(\dfrac{\tpr_1(\tpr_2+r-x)}{(t_1+\tpr_1)(t_2+\tpr_2)}-\dfrac{\tpr_2(\tpr_1+r-x)}{(t_1+\tpr_2)(t_2+\tpr_1)}\right).
\end{align*}
Observe that, as $x\geq r-(1-\ve)t_2$,
$$
\frac{\tpr_2+r-x}{t_2+\tpr_2}-\frac{\tpr_1+r-x}{t_2+\tpr_1}=\frac{(t_2-r+x)(\tpr_2-\tpr_1)}{(t_2+\tpr_2)(t_2+\tpr_1)}\geq\ve\frac{t_2(\tpr_2-\tpr_1)}{(t_2+\tpr_2)(t_2+\tpr_1)}.
$$
On the other hand, since $(t,\tpr)\longmapsto\dfrac{1}{t+\tpr}$ is TP$_2$ on $\left(\R_+^\ast\right)^2$,
$$
\frac{\tpr_1(\tpr_2+r-x)}{(t_1+\tpr_1)(t_2+\tpr_2)}\geq \frac{\tpr_1(\tpr_2+r-x)}{(t_1+\tpr_2)(t_2+\tpr_1)}
$$
and, as a consequence, 
$$
\frac{\tpr_1(\tpr_2+r-x)}{(t_1+\tpr_1)(t_2+\tpr_2)}-\frac{\tpr_2(\tpr_1+r-x)}{(t_1+\tpr_2)(t_2+\tpr_1)}\geq-\frac{(r-x)(\tpr_2-\tpr_1)}{(t_1+\tpr_2)(t_2+\tpr_1)}.
$$
Hence,
\begin{align*}
&C_{\mu^\ve}(t_2,\tpr_2,x)C_{\mu^\ve}(t_1,\tpr_1,x)-C_{\mu^\ve}(t_2,\tpr_1,x)C_{\mu^\ve}(t_1,\tpr_2,x)\\
&\geq\dfrac{\ve(r-x)t_2(\tpr_2-\tpr_1)}{t_2+\tpr_1}\left(\dfrac{t_2}{t_2+\tpr_2}-\dfrac{t_1}{t_1+\tpr_2}\right)\geq0
\end{align*}
since $t\longmapsto\dfrac{t}{t+\tpr_2}$ is non-decreasing on $\R_+$.\\
Suppose that $x\in[r-(1-\ve)t_1,r+\tpr_1[$. Then it follows from the total positivity property of the function $(t,\tpr)\longmapsto\dfrac{1}{t+\tpr}$ that
\begin{align*}
&C_{\mu^\ve}(t_1,\tpr_1,x)C_{\mu^\ve}(t_2,\tpr_2,x)-C_{\mu^\ve}(t_1,\tpr_2,x)C_{\mu^\ve}(t_2,\tpr_1,x)\\
&=t_1t_2(\tpr_1+r-x)(\tpr_2+r-x)\left(\frac{1}{(t_1+\tpr_1)(t_2+\tpr_2)}-\frac{1}{(t_1+\tpr_2)(t_2+\tpr_1)}\right)\geq0.
\end{align*}
Finally, if $x\in[r+\tpr_1,+\infty[$, the left-hand side of (\ref{eq:exaMRLMTP2}) equals zero and (\ref{eq:exaMRLMTP2}) is obviously satisfied. This ends the proof of (\ref{eq:exaMRLMTP2}). We mention that (\ref{eq:exaMRLMTP2}) remains true if one allows $t$, or $\tpr$, or both $t$
and $\tpr$ to take the value zero.
\end{proof}
\begin{rem}
We mention that $C_{\mu^\ve}$ does not have interval support. Indeed, if $(\mt_1,x_1)$ and $(\mt_2,x_2)$ satisfy $0<t_1<t_2$ and $x_1<\tpr_1+r<x_2<\tpr_2+r$, then $(\mt_1,x_1)\leq(\mt_1,x_2)\leq(\mt_2,x_2)$, $C_{\mu^\ve}(\mt_1,x_1)>0$, $C_{\mu^\ve}(\mt_2,x_2)>0$ and $C_{\mu^\ve}(\mt_1,x_2)=0$.
As a consequence, Proposition 2.1 of \citep{KaR80} does not apply.
\end{rem}
The next example shows that the family of MRL processes includes strictly that of processes which have MTP$_2$ integrated survival functions.
\begin{exa}\label{exa:coExaMRLMTP2}
Let $(\mu_\mt,\mt\in\R_+^2)$ be the process given by: for every $t\in\R_+$, $\mu_{(t,0)}=\mu_{(0,t)}=\delta_0$ and, for every $\mt=(t,\tpr)\in\left(\R_+^\ast\right)^2$,
$$
\mu_\mt=\frac{t+\tpr}{2t+\tpr}\delta_{-t}+\frac{t}{2t+\tpr}\delta_{t+\tpr}.
$$ 
Then $\left(\mu_\mt,\mt\in\R_+^2\right)$ is a MRL process whose integrated survival function $C_\mu$ is not TP$_2$ in $(t,\tpr)$ when $x$ is fixed. We start by showing that $\left(\mu_\mt,\mt\in\R_+^2\right)$ is a MRL process. The Hardy-Littlewood function of $\mu_\mt$ is given by:
$$
\forall\,x\in\R,\quad\Psi_{\mu_\mt}(x)=\left\{
\begin{array}{cl}
0&\text{if }x\leq-t,\\
t+\tpr&\text{if }-t<x\leq t+\tpr,\\
x&\text{if }x>t+\tpr
\end{array}
\right.
$$
and it is not difficult to verify that, for every $x\in\R$, $\mt\longmapsto\Psi_{\mu_\mt}(x)$ is non-decreasing. 

Now, we show that $C_\mu$ is not TP$_2$ in $(t,\tpr)$ when $x$ is fixed. We have
$$
\forall\,(\mt,x)\in\left(\R_+^\ast\right)^2\times\R,\quad C_\mu(\mt,x)=\left\{
\begin{array}{cl}
-x&\text{if }x<-t,\\
\dfrac{t(t+\tpr-x)}{2t+\tpr}&\text{if }-t\leq x<t+\tpr,\\
0&\text{if }x\geq t+\tpr
\end{array}
\right.
$$
and, in particular, $C_\mu(t,\tpr,0)=\dfrac{t(t+\tpr)}{2t+\tpr}$. But, $(t,\tpr)\longmapsto\dfrac{t+\tpr}{2t+\tpr}$ is not TP$_2$ on $\left(\R_+^\ast\right)^2$.
\end{exa}

\subsection{MRL processes obtained by censoring transformations}
Let $\left(\nu_t,t\in\R_+\right)$ be a one-parameter MRL process such that, for every $t\in\R_+$,
$$
r_t=r_{\nu_t}=\inf\{z\in\R,\,\nu_t([z,+\infty[)=0\}<\infty.
$$
Let $\phi,\varphi:\,\R_+^2\to\R$ be two maps such that $\phi$ is non-increasing, $\varphi$ is non-decreasing,   $\phi(0,0)=\varphi(0,0)=r_0$ and, for every $\mt\in\R_+^2\setminus\{(0,0)\}$, $\varphi(\mt)\geq r_t$ and $\varphi(\mt)>\phi(\mt)$. Consider the process $\left(\mu_\mt,\mt\in\R_+^2\right)$ given by: for every $t\in\R_+$, $\mu_{(0,0)}=\nu_0$ and 
\begin{equation}\label{eq:ExaMRLprocGnrle}
\forall\,\mt\in\R_+^2\setminus\{(0,0)\},\quad\mu_\mt=1_{]-\infty,\phi(\mt)[}\nu_t+\alpha_\mt\delta_{\phi(\mt)}+\beta_\mt\delta_{\varphi(\mt)},
\end{equation}
where
$$
\alpha_\mt=\frac{1}{\varphi(\mt)-\phi(\mt)}\int_{[\phi(\mt),r_t]}(\varphi(\mt)-y)\nu_t(dy)
$$
and
$$
\beta_\mt=\frac{1}{\varphi(\mt)-\phi(\mt)}\int_{[\phi(\mt),r_t]}(y-\phi(\mt))\nu_t(dy).
$$
Observe that $(\alpha_\mt,\beta_\mt)$ is the unique solution of the linear system:
$$
\alpha_\mt+\beta_\mt=\nu_t([\phi(\mt),r_t])\text{ and }\phi(\mt)\alpha_\mt+\varphi(\mt)\beta_\mt=\int_{[\phi(\mt),r_t]}y\nu_t(dy).
$$
\begin{prop}\label{prop:MRLprocGnrle}
The process $\left(\mu_\mt,\mt\in\R_+^2\right)$ defined by (\ref{eq:ExaMRLprocGnrle}) is a MRL process.
\end{prop}
\begin{proof}
For every $\mt\in\R_+^2$, the Hardy-Littlewood function $\Psi_{\mu_\mt}$ attached to $\mu_\mt$ is given by:
$$
\Psi_{\mu_\mt}(x)=\left\{
\begin{array}{cl}
\Psi_{\nu_t}(x)&\text{if }x\leq\phi(\mt),\\
\varphi(\mt)&\text{if }\phi(\mt)< x\leq \varphi(\mt),\\
x&\text{if }x> \varphi(\mt)
\end{array}
\right.
$$
from which one deduces that the family $\left(\Psi_{\mu_\mt},\mt\in\R_+^2\right)$ is pointwise non-decreasing.
\end{proof}
The integrated survival function of a process of the form (\ref{eq:ExaMRLprocGnrle}) is not necessarily TP$_2$ in $(t,\tpr)$ when $x$ is fixed. For instance, if one takes $\nu_t=\delta_t$, $\phi(\mt)=-t$ and $\varphi(\mt)=t+\tpr$, one recovers the process given in Example \ref{exa:coExaMRLMTP2} whose integrated survival function is not MTP$_2$. Now, we restrict ourselves to two-parameter MRL processes of the form (\ref{eq:ExaMRLprocGnrle}) which have MTP$_2$ integrated survival functions.
\begin{prop}\label{prop:ExaMRLMTP2Mzero}
Let $\nu$ denote an integrable probability measure such that the upper bound $r$ of its support is finite. Let $\left(\mu_\mt,\mt\in\R_+^2\right)$ be the process defined by: $\mu_{(0,t)}=\nu$ for every $t\in\R_+$, and
\begin{equation}\label{eq:ExaMRLMTP2Mzero}
\forall\,\mt\in\R_+\times\R_+^{\ast},\quad\mu_\mt=1_{]-\infty,r-t[}\nu+\alpha_\mt\delta_{r-t}+\beta_\mt\delta_{r+\tpr},
\end{equation}
where
\begin{equation}\label{eq:ExaMRL0DefAlpha}
\alpha_\mt=\frac{1}{t+\tpr}\int_{[r-t,r]}(r+\tpr-y)\nu(dy)
\end{equation}
and
\begin{equation}\label{eq:ExaMRL0DefBeta}
\beta_\mt=\frac{1}{t+\tpr}\int_{[r-t,r]}(y-r+t)\nu(dy).
\end{equation}
Then the integrated survival function $C_\mu$ of $\left(\mu_\mt,\mt\in\R_+^2\right)$ is MTP$_2$.
\end{prop}
We omit the proof of Proposition \ref{prop:ExaMRLMTP2Mzero} since this result is a particular case of Proposition \ref{prop:ExaMTP2ModifNu} stated  below.
Processes of the form (\ref{eq:ExaMRLMTP2Mzero}) have constant mean. These processes may be slightly modified so that the resulting processes still have MTP$_2$ integrated survival functions, but have means which depend on $\mt$. For instance, the process $\left(\mu_\mt^\ve,\mt\in\R_+^2\right)$ in Proposition \ref{prop:TheoMRLMTP2} is a modification of the process $\left(\mu_\mt,\mt\in\R_+^2\right)$ given by (\ref{eq:ExaMRLMTP2Mzero}) when $\nu=\delta_r$, $r\in\R$.
In the next result, we provide a modified version of the process $\left(\mu_\mt,\mt\in\R_+^2\right)$ given by (\ref{eq:ExaMRLMTP2Mzero}) which has MTP$_2$ integrated survival function, but whose mean depends on $\mt$.
\begin{prop}\label{prop:ExaMTP2ModifNu}
Let $\nu$ be an integrable probability measure whose support has a finite upper bound denoted by $r$. For every $\mt\in\R_+^2$, let $\alpha_\mt$ and $\beta_\mt$ be given by (\ref{eq:ExaMRL0DefAlpha}) and (\ref{eq:ExaMRL0DefBeta}) respectively. Then, for every $\ve\geq0$, the process $\left(\mu_\mt^\ve,\mt\in\R_+^2\right)$ defined by: for every $t\in\R_+$, $\mu^\ve_{(0,t)}=\nu$, and
\begin{equation}\label{eq:ExaMTP2ModifNu}
\forall\,\mt\in\R_+^\ast\times\R_+,\quad\mu_\mt^\ve=1_{]-\infty,r-t[}\nu+\alpha_\mt\delta_{r-t}+\beta_\mt\delta_{r+(1+\ve)\tpr}
\end{equation}
has MTP$_2$ integrated survival function.
\end{prop}
\begin{proof}
We start by showing that $\left(\mu_\mt^\ve,\mt\in\R_+^2\right)$ is a MRL process. Let $\Psi_\nu$ and $C_\nu$ denote the Hardy-Littlewood and the integrated survival functions of $\nu$ respectively. Let $\wtb$ be the function defined by  
\begin{align*}
\forall\,t\in\R_+,\quad\wtb(t)=C_\nu(r-t)=\int_{[r-t,r]}(y-r+t)\nu(dy).
\end{align*}
Since $C_\nu$ is nonnegative and convex, $t\longmapsto\dfrac{\wtb(t)}{t}$ is nonnegative and non-decreasing on $\R_+^\ast$. Moreover, $(t,\tpr)\longmapsto\dfrac{t\tpr}{t+\tpr}$ is nonnegative, non-decreasing, supermodular and TP$_2$ on $\R_+^\ast\times\R_+^\ast$. Then 
$$
(t,\tpr)\longmapsto\frac{\tpr\wtb(t)}{t+\tpr}\text{ is non-decreasing, supermodular and TP}_2\text{ on }\R_+^\ast\times\R_+^\ast.
$$
For every $\mt\in\R_+^2$, we denote by $\Psi_{\mu_\mt^\ve}$ the Hardy-Littlewood function of $\mu_\mt^\ve$. Then, by observing that $\wtb(t)=(t+\tpr)\beta_\mt$, we have, for every $x\in\R$,
$$
\Psi_{\mu_\mt^\ve}(x)=\left\{
\begin{array}{ll}
\Psi_\nu(x)+\dfrac{\ve\tpr\wtb(t)}{(t+\tpr)\nu([x,r])}&\text{if }
x\leq r-t,\\
r+(1+\ve)\tpr&\text{if }r-t<x\leq r+(1+\ve)\tpr,\\
x&\text{if }x>r+(1+\ve)\tpr.
\end{array}
\right.
$$
Let $\mt_1\leq\mt_2$ be elements of $\left(\R_+^\ast\right)^2$. We have
$$
\R=]-\infty,r-t_2]\cup]r-t_2,r-t_1]\cup]r-t_1,r+(1+\ve)\tpr_1] \cup]r+(1+\ve)\tpr_1,+\infty[.
$$
Let $x$ be a real number. If $x\in]-\infty,r-t_2]$, then $\Psi_{\mu_{\mt_1}^\ve}(x)\leq \Psi_{\mu_{\mt_2}^\ve}(x)$ since $(t,\tpr)\longmapsto\dfrac{\tpr\wtb(t)}{t+\tpr}$ is non-decreasing on $\R_+^\ast\times\R_+^\ast$. If $x\in]r-t_2,r-t_1]$, then, as $\wtb(t_1)\leq t_1\nu([r-t_1,r])$ and $\Psi_\nu(x)\leq r$,
$$
\Psi_{\mu_{\mt_1}^\ve}(x)=\Psi_\nu(x)+\dfrac{\ve\tpr_1\wtb(t_1)}{(t_1+\tpr_1)\nu([x,r])}\leq r+\ve\tpr_1\leq r+(1+\ve)\tpr_2=\Psi_{\mu_{\mt_2}^\ve}(x).
$$
If $x\in]r-t_1,r+(1+\ve)\tpr_1]$, then $\Psi_{\mu_{\mt_1}^\ve}(x)=r+(1+\ve)\tpr_1\leq r+(1+\ve)\tpr_2=\Psi_{\mu_{\mt_2}^\ve}(x)$. If $x>r+(1+\ve)\tpr_1$, then, by definition of $\Psi_{\mu_{\mt_2}^\ve}$, $\Psi_{\mu_{\mt_1}^\ve}(x)=x\leq\Psi_{\mu_{\mt_2}^\ve}(x)$. Thus, $\left(\mu_\mt^\ve,\mt\in\R_+^2\right)$ is a MRL process which implies that the integrated survival function $C_{\mu^\ve}$ of $\left(\mu_\mt^\ve,\mt\in\R_+^2\right)$ defined on $\left(\R_+^\ast\right)^2$ by 
\begin{equation}\label{eq:DefDTailMTP2ModifNu}
C_{\mu^\ve}(t,\tpr,x)=\left\{
\begin{array}{ll}
C_\nu(x)+\dfrac{\ve\tpr\wtb(t)}{t+\tpr}&\text{if }x<r-t,\\
\dfrac{\wtb(t)(r+(1+\ve)\tpr-x)}{t+\tpr}&\text{if }r-t\leq x<r+(1+\ve)\tpr,\\
0&\text{if }x\geq r+(1+\ve)\tpr
\end{array}
\right.
\end{equation}
is TP$_2$ in $(t,x)$ when $\tpr$ is fixed, and TP$_2$ in $(\tpr,x)$ when $t$ is fixed. Let us prove that $C_{\mu^\ve}$ is also TP$_2$ in $(t,\tpr)$ when $x$ is fixed which, by Theorem \ref{theo:MRLMTP2Tail}, entails that $C_{\mu^\ve}$ is MTP$_2$. 

We fix $t_1\leq t_2$, $\tpr_1\leq\tpr_2$, and $x\in\R$. Suppose that $x\in]-\infty,r-t_2[$. Since the function $(t,\tpr)\longmapsto\dfrac{\tpr\wtb(t)}{t+\tpr}$ is supermodular and TP$_2$ on $\R_+^\ast\times\R_+^\ast$, then
$$
(t,\tpr)\longmapsto C_\nu(x)+\dfrac{\ve\tpr\wtb(t)}{t+\tpr}\text{ is TP}_2\text{ on }\R_+^\ast\times\R_+^\ast,
$$
and, as in Proposition \ref{prop:TheoMRLMTP2}, we show that
\begin{equation}\label{eq:MTP2TailModif}
C_{\mu^\ve}(t_1,\tpr_1,x)C_{\mu^\ve}(t_2,\tpr_2,x)-C_{\mu^\ve}(t_1,\tpr_2,x)C_{\mu^\ve}(t_2,\tpr_1,x)\geq0.
\end{equation}
If $x\in[r-t_2,r-t_1[$, then
\begin{align*}
&C_{\mu^\ve}(t_1,\tpr_1,x)C_{\mu^\ve}(t_2,\tpr_2,x)-C_{\mu^\ve}(t_1,\tpr_2,x)C_{\mu^\ve}(t_2,\tpr_1,x)\\
&=\left(C_\nu(x)+\frac{\ve\tpr_1\wtb(t_1)}{t_1+\tpr_1}\right)\left(\dfrac{\wtb(t_2)(r+(1+\ve)\tpr_2-x)}{t_2+\tpr_2}\right)- \\
&\qquad\left(C_\nu(x)+\frac{\ve\tpr_2\wtb(t_1)}{t_1+\tpr_2}\right)\left(\dfrac{\wtb(t_2)(r+(1+\ve)\tpr_1-x)}{t_2+\tpr_1}\right)\\
&=C_\nu(x)\wtb(t_2)\left(\dfrac{r+(1+\ve)\tpr_2-x}{t_2+\tpr_2}-\dfrac{r+(1+\ve)\tpr_1-x}{t_2+\tpr_1}\right)+\\
&\qquad\ve\wtb(t_1)\wtb(t_2)\left(\frac{\tpr_1(r+(1+\ve)\tpr_2-x)}{(t_2+\tpr_2)(t_1+\tpr_1)}-\frac{\tpr_2(r+(1+\ve)\tpr_1-x)}{(t_1+\tpr_2)(t_2+\tpr_1)}\right).
\end{align*}
Since $x\geq r-t_2$, and since $t\longmapsto\dfrac{t}{t+\tpr_2}$ is non-decreasing on $\R_+$, we have
\begin{align*}
&\dfrac{r+(1+\ve)\tpr_2-x}{t_2+\tpr_2}-
\dfrac{r+(1+\ve)\tpr_1-x}{t_2+\tpr_1}\\
&=\frac{((1+\ve)t_2-r+x)(\tpr_2-\tpr_1)}{(t_2+\tpr_2)(t_2+\tpr_1)}\geq\frac{\ve t_2(\tpr_2-\tpr_1)}{(t_2+\tpr_2)(t_2+\tpr_1)} \geq\frac{\ve t_1(\tpr_2-\tpr_1)}{(t_1+\tpr_2)(t_2+\tpr_1)}.
\end{align*}
Moreover, the TP$_2$ property of $(t,\tpr)\longmapsto\dfrac{1}{t+\tpr}$ on $\R_+^\ast\times\R_+^\ast$ yields
$$
\frac{\tpr_1(r+(1+\ve)\tpr_2-x)}{(t_1+\tpr_1)(t_2+\tpr_2)}\geq \frac{\tpr_1(r+(1+\ve)\tpr_2-x)}{(t_1+\tpr_2)(t_2+\tpr_1)}.
$$
Hence,
\begin{align*}
&\frac{\tpr_1(r+(1+\ve)\tpr_2-x)}{(t_2+\tpr_2)(t_1+\tpr_1)}-\frac{\tpr_2(r+(1+\ve)\tpr_1-x)}{(t_1+\tpr_2)(t_2+\tpr_1)}\\
&\geq \frac{\tpr_1(r+(1+\ve)\tpr_2-x)}{(t_1+\tpr_2)(t_2+\tpr_1)}-\frac{\tpr_2(r+(1+\ve)\tpr_1-x)}{(t_1+\tpr_2)(t_2+\tpr_1)}=-\frac{(r-x)(\tpr_2-\tpr_1)}{(t_1+\tpr_2)(t_2+\tpr_1)}.
\end{align*}
Thus, since $x\leq r-t_1$, and since $t\longmapsto\dfrac{\wtb(t)}{t}$ is non-decreasing on $\R_+^\ast$,
\begin{align*}
&C_{\mu^\ve}(t_1,\tpr_1,x)C_{\mu^\ve}(t_2,\tpr_2,x)-C_{\mu^\ve}(t_1,\tpr_2,x)C_{\mu^\ve}(t_2,\tpr_1,x)\\
&\geq\frac{\ve t_1(r-x)\wtb(t_2)(\tpr_2-\tpr_1)}{(t_1+\tpr_2)(t_2+\tpr_1)}\left(\frac{C_\nu(x)}{r-x}-\frac{\wtb(t_1)}{t_1}\right)\\
&=\frac{\ve t_1(r-x)\wtb(t_2)(\tpr_2-\tpr_1)}{(t_1+\tpr_2)(t_2+\tpr_1)}\left(\frac{\wtb(r-x)}{r-x}-\frac{\wtb(t_1)}{t_1}\right)\geq0.
\end{align*}
If $x\in[r-t_1,r+(1+\ve)\tpr_1[$, then, as $x\in[r-t_2,r+(1+\ve)\tpr_2[$ and as
$$
(t,\tpr)\longmapsto\dfrac{\wtb(t)(r+(1+\ve)\tpr-x)}{t+\tpr}\text{ is TP}_2\text{ on }\R_+^\ast\times\R_+^\ast,
$$ Inequality (\ref{eq:MTP2TailModif}) still holds. \\
Finally, if $x\geq r+(1+\ve)\tpr_1$, then (\ref{eq:MTP2TailModif}) still holds since $C_{\mu^\ve}(t_1,\tpr_1,x)=0=C_{\mu^\ve}(t_2,\tpr_1,x)$.\\
Moreover, (\ref{eq:MTP2TailModif}) remains valid if $t_1$, or $\tpr_1$, or both $t_1$ and $\tpr_1$ equal zero.
\end{proof}
\begin{rem}
Note that the integrated survival function $C_{\mu^\ve}$ given by (\ref{eq:DefDTailMTP2ModifNu}) does not have interval support, and then Proposition 2.1 of \citep{KaR80} does not apply to $C_{\mu^\ve}$. Indeed, if $(\mt_1,x_1)$ and $(\mt_2,x_2)$ satisfy $0<t_1<t_2$ and $x_1<\tpr_1+(1+\ve)r<x_2<\tpr_2+(1+\ve)r$, then $(\mt_1,x_1)\leq(\mt_1,x_2)\leq(\mt_2,x_2)$, $C_{\mu^\ve}(\mt_1,x_1)>0$, $C_{\mu^\ve}(\mt_2,x_2)>0$ and $C_{\mu^\ve}(\mt_1,x_2)=0$.
\end{rem} 
\begin{rem}
For every $\ve\in(0,1)$, let $\left(\mu_\mt^\ve,\mt\in\R_+^2\right)$ be the process given by (\ref{eq:ExaMTP2ModifNu}). 
\item[1.] Let $T_{\mu_\mt^\ve}$ denote the Cox-Hobson solution to the Skorokhod embedding problem for $\mu_\mt^\ve$. If $(B_v,v\geq0)$ is a standard Brownian motion started at $0$, then $\left(B_{T_{\mu_\mt^\ve}},\mt\in\R_+^2\right)$ is a two-parameter submartingale associated to $\left(\mu_\mt^\ve,\mt\in\R_+^2\right)$.
\item[2.] If $\nu^\ve=\left(\nu^\ve_\mt,\mt\in[0,1]\times\R_+\right)$ is the process given by
$$
\forall\,\mt\in\R^2_+,\quad\nu^\ve_\mt=(1-t)\mu^\ve_{(1,\tpr)}+t\mu^\ve_{(2,\tpr)},
$$
then $(\nu^\ve_\mt,\mt\in\R^2_+)$ has MTP$_2$ integrated survival function.
\end{rem}
The interest of the MTP$_2$ property relies on the fact that there are several transformations which preserve this property. Therefore, from a given process with MTP$_2$ integrated survival function, one may generate several other processes which satisfy the same MTP$_2$ condition. For instance, the Point 2. of the preceding remark provides a transformation which preserves the MTP$_2$ property.
\subsection{MRL processes obtained by subordination}
We exhibit several subordinated processes with MTP$_2$ integrated survival functions. Precisely, from a given MTP$_2$ integrated survival function, we generate  many other integrated survival functions using a well-known composition formula. We exploit total positivity properties of certain $\R_+$-valued Markov processes. The following results taken from \citep{Ka64} concerns $\R_+$-valued Markov processes which have TP$_2$ transition kernels.
\begin{theorem}\cite[\text{Theorem 5.2.}]{Ka64}\label{theo:TP2MarkovCS}
Let $(p_t,t\in\R_+^\ast)$ be the transition densities of a right-continuous time-homogeneous $\R_+$-valued Markov process started at zero. Suppose that
\begin{equation}\label{eq:TP2MarkovCS}
\forall\,t\in\R_+^\ast,\quad p_t\text{ is continuous and TP}_2\text{ on }\R_+\times\R_+.
\end{equation}
Then $(t,\lambda)\longmapsto p_t(0,\lambda)$ and $(t,\lambda)\longmapsto p_t(\lambda,0)$ are TP$_2$ on $\R_+^\ast\times\R_+$.
\end{theorem}
There are many time-homogeneous Markov processes with TP$_2$ transition densities. For instance, the transition densities (with respect to the Lebesgue measure) of one-dimensional diffusions are TP$_2$. Here is an analog of Theorem \ref{theo:TP2MarkovCS} for continuous-time Markov chains which have TP$_2$ transition matrices. We mention that this class of Markov chains includes birth-and-death processes.
\begin{theorem}\label{theo:TP2MarkovChainCS}\citep[\text{Theorem 4.3, Points (i) and (ii).}]{Ka64} Let $(P_t,t\in\R_+)$ be the transition matrices of a right-continuous time-homogeneous $\R_+$-valued Markov chain issued from zero. If
\begin{equation}\label{eq:TP2MarkovChainCS}
\forall\,t\geq0,\quad (i,j)\longmapsto P_t(i,j)\text{ is TP}_2\text{ on }\N\times\N,
\end{equation}
then $(t,i)\longmapsto P_t(0,i)$ and $(t,i)\longmapsto P_t(i,0)$ are TP$_2$ on $\R_+\times\N$.
\end{theorem}
We now provide a way to generate many integrable processes which have MTP$_2$ integrated survival functions.
\begin{theorem}\label{theo:MTP2TailSub}
Let $\left(\mu_{(\lambda,\lambda')},(\lambda,\lambda')\in\R_+^2\right)$ be an integrable process which has MTP$_2$ integrated survival function denoted by $C_{\mu}$. Suppose that there exists a positive constant $K$ satisfying:
\begin{equation}\label{eq:MTP2TailIntCond}
\forall\,(\lambda,\lambda')\in\R_+^2,\quad\int_{\R}|y|\mu_{(\lambda,\lambda')}(dy)\leq K(1+\lambda)(1+\lambda').
\end{equation}
\item[1.]Let $\left(p_t,t\in\R_+^\ast\right)$ and $\left(q_t,t\in\R_+^\ast\right)$ be the transition densities of two right-continuous, integrable and time-homogeneous $\R_+$-valued Markov processes issued from zero. Suppose that, for every $t\geq0$, $p_t$ and $q_t$ satisfy (\ref{eq:TP2MarkovCS}). Then the process $\left(\sigma_\mt,\mt\in(\R_+^\ast)^2\right)$ given by:
$$
\forall\,(\mt,x)\in(\R_+^\ast)^2\times\R,\quad\sigma_\mt([x,+\infty[)=\int\!\!\!\int_{\R_+^2}\mu_{(\lambda,\lambda')}([x,+\infty[)p_t(0,\lambda)q_{\tpr}(0,\lambda')d\lambda d\lambda'
$$
is integrable, and it has a MTP$_2$ integrated survival function.
\item[2.]Let $(P_t,t\in\R_+)$ and $(Q_t,t\in\R_+)$ be the transition matrices of two right-continuous, integrable and time-homogeneous $\R_+$-valued Markov chains started at zero. If, for every $t\geq0$, $P_t$ and $Q_t$ fulfill Condition (\ref{eq:TP2MarkovChainCS}), then the process $\left(\sigma_\mt,\mt\in\R_+^2\right)$ defined by:
$$
\forall\,(\mt,x)\in\R_+^2\times\R,\quad\sigma_\mt([x,+\infty[)=\sum\limits_{i\in\N}\sum\limits_{j\in\N}\mu_{(i,j)}([x,+\infty[)P_t(0,i)Q_{\tpr}(0,j)
$$ 
is integrable, and its integrated survival function is MTP$_2$.
\end{theorem}
To prove Theorem \ref{theo:MTP2TailSub}, we need the following classical result which is an immediate consequence of Proposition \ref{prop:CompoMTP2}.
\begin{corol}\label{corol:CompoMTP2}
Let $f:\,\R_+^2\times\R\to\R_+$ be MTP$_2$ and $g:\,\R_+^2\to\R_+$ be TP$_2$. Let $\rho$ be a $\sigma$-finite positive measure on $\R_+$ such that:
$$
\forall\,(\mt,x)\in\R_+^2\times\R,\quad\int_{\R_+}f(u,t,x)g(u,\tpr)\rho(du)\text{ is finite.}
$$
Then the function $h$ defined on $\R_+^2\times\R$ by
$$
h(\mt,x)=\int_{\R_+}f(u,t,x)g(u,\tpr)\rho(du)
$$
is MTP$_2$ on $\R_+^2\times\R$.
\end{corol}
\begin{proof}[Proof of Theorem \ref{theo:MTP2TailSub}.]We prove only Point 1 since the proof of Point 2 is quite similar.\\ 
Since $C_\mu$
is TP$_2$ in $(\lambda,x)$ when $\lambda'$ is fixed, then, for every $(\lambda',x)$, $\lambda\longmapsto C_\mu(\lambda,\lambda',x)$ is non-decreasing, and then it is Borel-measurable. We deduce that, for every fixed $(\lambda',x)\in\R_+\times\R$, the function $\lambda\longmapsto\mu_{(\lambda,\lambda')}([x,+\infty[)$ is Borel-measurable. Indeed, one has
$$
\mu_{(\lambda,\lambda')}([x,+\infty[)=\lim\limits_{n\to+\infty} n\left(C_\mu\left(\lambda,\lambda',x-\frac{1}{n}\right)-C_\mu(\lambda,\lambda',x)\right).
$$
Then one may define the  process $\left(\eta_{(t,\lambda')},(t,\lambda')\in\R_+\times\R_+^\ast\right)$ by:
$$
\forall\,(t,\lambda',x)\in\R_+^\ast\times\R_+\times\R,\quad\eta_{(t,\lambda')}([x,+\infty[)=\int_{\R_+}\mu_{(\lambda,\lambda')}([x,+\infty[)p_t(0,\lambda)d\lambda.
$$
Moreover, as the Markov processes we deal with are integrable, the Fubini's theorem and Condition (\ref{eq:MTP2TailIntCond}) ensure that the process $\left(\eta_{(t,\lambda')},(t,\lambda')\in\R_+\times\R_+^\ast\right)$ is integrable and that its integrated survival function $C_\eta$ is given by:
$$
C_\eta(t,\lambda',x)=\int_{\R_+}C_{\mu}(\lambda,\lambda',x)p_t(0,\lambda)d\lambda.
$$
Because the function $(t,\lambda)\to p_t(0,\lambda)$ is TP$_2$ on $\R_+^\ast\times\R_+$, and because $C_{\mu}$ is MTP$_2$ on $\R_+^2\times\R$, we deduce from Corollary \ref{corol:CompoMTP2} that $C_\eta$ is MTP$_2$ on $\R_+^\ast\times\R_+\times\R$. Similarly, one shows that:
$$
\forall\,(t,x)\in\R_+^\ast\times\R,\quad \lambda'\longmapsto\eta_{(t,\lambda')}([x,+\infty[)\text{ is Borel-measurable,}
$$
and  defines the process $\left(\sigma_\mt,\mt\in(\R_+^\ast)^2\right)$ by:
$$
\forall\,(\mt,x)\in(\R_+^\ast)^2\times\R,\quad\sigma_{\mt}([x,+\infty[)=\int_{\R_+}\eta_{(t,\lambda')}([x,+\infty[)q_{\tpr}(0,\lambda')d\lambda'.
$$
Then it follows from Condition (\ref{eq:MTP2TailIntCond}) and from Fubini's theorem that $\left(\sigma_\mt,\mt\in(\R_+^\ast)^2\right)$ is integrable and that its integrated survival function $C_\sigma$ is given by:
$$
\forall\,(\mt,x)\in(\R_+^\ast)^2\times\R,\quad C_\sigma(\mt,x)=\int_{\R_+}C_\eta(t,\lambda',x)q_{\tpr}(0,\lambda')d\lambda'.
$$
Hence, since $(\tpr,\lambda')\longmapsto q_{\tpr}(0,\lambda')$ is TP$_2$ in $\R_+^\ast\times\R$ and since $C_{\eta}$ is MTP$_2$ on $\R_+^\ast\times\R_+\times\R$, then, by Corollary \ref{corol:CompoMTP2}, $C_\sigma$ is MTP$_2$ on $(\R_+^\ast)^2\times\R$. 
\end{proof}
\begin{exa}
Let $\ve\in(0,1)$ and let $\nu$ be an integrable process whose support admits a finite upper bound. Then the  process $\left(\mu_\mt^\ve,\mt\in\R_+^2\right)$ given in Proposition \ref{prop:ExaMTP2ModifNu} satisfies Condition (\ref{eq:MTP2TailIntCond}) of Theorem \ref{theo:MTP2TailSub}. Indeed, one has:
$$
\forall\,\mt\in\R_+^2,\quad\int_{\R_+}|y|\mu_\mt^\ve(dy)\leq K(1+t+\tpr),
$$
where one may choose $K=\max\left\{2,|r|+\int_{]-\infty,r]}|y|\nu(dy) \right\}$. As a consequence:
\item[1.]  If $(p_t,t\in\R_+^\ast)$ and $(q_t,t\in\R_+^\ast)$ are the transition densities of two right-continuous, integrable and time-homogeneous $\R_+$-valued Markov processes issued from zero, and if, for every $t\geq0$, $p_t$ and $q_t$ satisfy Condition (\ref{eq:TP2MarkovCS}), then the process $\left(\sigma_\mt^\ve,\mt\in(\R_+^\ast)^2\right)$ given by:
$$
\forall\,(\mt,x)\in(\R_+^\ast)^2\times\R,\quad\sigma_\mt([x,+\infty[)=\int\!\!\!\int_{\R_+^2}\mu_{(\lambda,\lambda')}([x,+\infty[)p_t(0,\lambda)q_{\tpr}(0,\lambda')d\lambda d\lambda'
$$
has a MTP$_2$ integrated survival function.
\item[2.]If $p$ and $q$ denote two TP$_2$ families of density functions on $\N$ such that, for every $n\in\N$, the sums $\sum_{i\in\N}ip(n,i)$ and $\sum_{i\in\N}iq(n,i)$ are finite, then the process $\left(\sigma_{(n,m)}^\ve,(n,m)\in\N^2\right)$ defined by:
$$
\forall\,(n,m,x)\in\N\times\N\times\R,\quad\sigma^\ve_{(n,m)}([x,+\infty[)=\sum\limits_{i\in\N}\sum\limits_{j\in\N}\mu^\ve_{(i,j)}([x,+\infty[)p(n,i)q(m,j)
$$
has MTP$_2$ integrated survival function. Here are some examples of such TP$_2$ families of density functions taken from \cite[Section 8]{Ka64}.
\begin{enumerate}
\item[(i)]Let $a\in(0,1)$ and let $p^{(a)}$ denote the family of binomial densities given by:
$$
\forall\,(n,i)\in\N\times\N,\quad p^{(a)}(n,i)=\left(
\begin{array}{c}
n\\i
\end{array}
\right)a^i(1-a)^{n-i},
$$
where $\left(\begin{array}{c}n\\i\end{array}\right)$, $(n,i)\in\N\times\N$ are the usual binomial coefficients. Then $p^{(a)}$ is TP$_2$ on $\N\times\N$.
\item[(ii)]For every $a\in(0,1)$, the function $p^{(a)}$ defined by:
$$
\forall\,(n,i)\in\N\times\N,\quad p^{(a)}(n,i)=\left(
\begin{array}{c}
n+i-1\\i
\end{array}
\right)a^n(1-a)^{i},
$$
is TP$_2$ on $\N\times\N$.
\item[(iii)]Let $a\in(0,1)$ and let $q^{(a)}$ be the function given by:
$$
\forall\,(n,i)\in\N\times\N,\quad q^{(a)}(n,i)=\left[
\begin{array}{c}
n\\i
\end{array}
\right]\frac{a^{(i^2+i)/2}}{\prod\limits_{l=1}^n(1+a^l)},
$$
where
$$
\left[
\begin{array}{c}
n\\i
\end{array}
\right]=\left\{
\begin{array}{cl}
1&\text{if }n=i=0,\\
\dfrac{(1-a^n)(1-a^{n-i})\cdots(1-a^{n-i+1})}{(1-a^i)(1-a^{i-1})\cdots(1-a)}&\text{if }1\leq i\leq n,\\
0&\text{otherwise.}
\end{array}
\right.
$$
Then $q^{(a)}$ is TP$_2$ on $\N\times\N$.
\end{enumerate}
\end{exa}

\subsection{MRL processes obtained by convolution} 
One may also generate many MRL processes using convolution transformations.
\begin{prop}\label{prop:ExaMRLMTPConv}
 Let $\left(\mu_\mt,\mt\in\R_+^2\right)$ be a MRL process, and let $f$ be a log-concave and positive density function which admits a finite first order moment. Then the process $\left(\xi_\mt,\mt\in\R_+^2\right)$ defined by:
$$
\forall\,\mt\in\R_+^2,\quad\xi_\mt(dy)=\left(\int_{\R}f(y-z)\mu_\mt(dz)\right)dy
$$
is a MRL process.
\end{prop}
\begin{proof}
Let $C_\mu$ and $C_\xi$ be the integrated survival functions of $\left(\mu_\mt,\mt\in\R_+^2\right)$ and $\left(\xi_\mt,\mt\in\R_+^2\right)$ respectively. By Fubini's theorem, we have:
\begin{align*}
C_{\xi}(\mt,x)&=\int_{\R}1_{[x,+\infty[}(v)(v-x)\xi_\mt(dv)=\int_{\R}1_{[x,+\infty[}(v)(v-x)\left(\int_{\R}f(v-y)\mu_\mt(dy)\right)dv\\
&=\int_{\R}\left(\int_{\R}1_{[x,+\infty[}(v)(v-x)f(v-y)dv\right)\mu_\mt(dy).
\end{align*}
By the change of variable $v-y=x-z$, we obtain:
\begin{align*}
C_{\xi}(\mt,x)&=\int_{\R}\left(\int_{\R}1_{]-\infty,y]}(z)(y-z)f(x-z)dz\right)\mu_\mt(dy)\\
&=\int_{\R}\left(\int_{\R}1_{[z,+\infty[}(y)(y-z)f(x-z)dz\right)\mu_\mt(dy)\\
&=\int_{\R}\left(\int_{\R}1_{[z,+\infty[}(y)(y-z)\mu_\mt(dy)\right)f(x-z)dz=
\int_{\R}C_\mu(\mt,z)f(x-z)dz.
\end{align*}
Since $f$ is log-concave, $(x,z)\longmapsto f(x-z)$ is TP$_2$ on $\R^2$. Moreover, as $\left(\mu_\mt,\mt\in\R_+^2\right)$ is a MRL process, then, by Theorem \ref{theo:MRLTPP2}, $C_\mu$ satisfies Condition (\ref{eq:MRLTPP2}). Hence, it follows from Corollary \ref{corol:CompoMTP2} that $C_\xi$ also satisfies Condition (\ref{eq:MRLTPP2}) meaning that $\left(\xi_\mt,\mt\in\R_+^2\right)$ is a MRL process.
\end{proof}
\begin{exa}
Let $\left(\mu_\mt,\mt\in\R_+^2\right)$ be the process defined by (\ref{eq:ExaMRLprocGnrle}). Then, for every log-concave density function $f$ which admits a finite first order moment, the process $\left(\xi_\mt,\mt\in\R_+^2\right)$ given by:
$$
\forall\,\mt\in\R_+^2,\quad\xi_\mt^\ve(dy)=\left(\int_{\R}f(y-z)\mu_\mt^\ve(dz)\right)dy
$$
is a MRL process.
\end{exa}
The preceding results are helpful to construct an associated two-parameter submartingale to certain non-MRL processes. We mention that there does not yet exist a counterpart of Kellerer's theorem for two-parameter processes. Recently, Juillet \citep{Ju16} proved that the Kellerer's theorem established in the one-parameter case does not extend to the two-parameter case.

\section{Construction of associated submartingales to a class of non-MRL processes}
We show that the Cox-Hobson algorithm yields an associated submartingale to certain non-MRL processes.
\subsection{Censoring transformed processes}
Let $\ve\in\R_+$ and let $\nu$ be an integrable probability measure. For every real number $r$ such that $\nu(]-\infty,r])$ and $\nu(]r,+\infty[)$ are positive, we wish to construct an associated submartingale to the process $\left(\mu^\ve_{\mt},\mt\in\R_+^2\right)$ given by: $\mu^\ve_{(0,0)}=\nu$ and, for every $\mt\in\R_+^2\setminus\{(0,0)\}$,
\begin{equation}\label{eq:ExaNonMRLMTP2}
\mu^\ve_\mt=(1_{]-\infty,r-t[}+1_{]r+\tpr,+\infty[})\nu+\alpha_\mt \delta_{r-t}+\beta_\mt^+\delta_{r+\tpr}+ \beta_\mt^- \delta_{r+(1+\ve)\tpr},
\end{equation}
where
\begin{align*}
&\beta_\mt^-=\frac{1}{t+\tpr}\int_{[r-t,r]}(y-r+t)\nu(dy),\text{ }\beta_\mt^+=\frac{1}{t+\tpr}\int_{]r,r+\tpr]}(y-r+t)\nu(dy),\\
&\alpha_\mt=\frac{1}{t+\tpr}\int_{[r-t,r+\tpr]}(r+\tpr-y)\nu(dy).
\end{align*}
\begin{rem}\text{}
\item[1.]The process $\left(\mu^\ve_\mt,\mt\in\R_+^2\right)$ given by (\ref{eq:ExaNonMRLMTP2}) is not MRL ordered in general. For instance, if $\ve=0=r$, the Hardy-Littlewood function of $\Psi_\mt^0$ of the process
\begin{equation}\label{eq:ExaNonMRLMuzero}
\mu_\mt^0=\left\{
\begin{array}{ll}
\nu&\text{if }\mt=(0,0),\\
(1_{]-\infty,-t[}+1_{]\tpr,+\infty})\nu+\alpha_\mt\delta_{-t}+\left(\beta^+_\mt+\beta^-_\mt\right)\delta_{\tpr}&\text{if }\mt\in\R_+^2\setminus\{(0,0)\}
\end{array}
\right.
\end{equation}
is given by: 
$$
\forall\,(\mt,x)\in\R_+^2\times\R,\quad\Psi_\mt^0(x)=\left\{
\begin{array}{ll}
\dfrac{\tpr C_\nu(-t)+tC_\nu(\tpr)}{C_\nu(-t)-C_\nu(\tpr)}&\text{if }(t,\tpr)\neq(0,0)\text{ and }-t<x\leq \tpr,\\
\Psi_\nu(x)&\text{otherwise,}
\end{array}
\right.
$$
where $C_\nu$ and $\Psi_\nu$ are the integrated survival and the Hardy-Littlewood functions of $\nu$ respectively. Take $0<t_1<t_2$, $\tpr\in\R_+$ and $x\in]-t_1,\tpr[$, and suppose that the survival function $\overline{\nu}$ of $\nu$ is continuous and decreasing. Then the function $t\longmapsto\Psi_{(t,\tpr)}^0(x)$ is decreasing on $[t_1,t_2]$ and, as a consequence, $\Psi^0_{(t_1,\tpr)}>\Psi_{(t_2,\tpr)}^0$. Indeed, $t\longmapsto\Psi_{(t,\tpr)}^0(x)$ is differentiable on $]t_1,t_2[$ and
\begin{align*}
\frac{\partial\Psi^0_{(t,\tpr)}(x)}{\partial t}&=\frac{C_\nu(\tpr)[C_\nu(-t)-C_\nu(\tpr)-(t+\tpr)\overline{\nu}(-t)]}{(C_\nu(-t)-C_\nu(\tpr))^2}\\
&=\frac{C_\nu(\tpr)}{(C_\nu(-t)-C_\nu(\tpr))^2}\left(\int_{-t}^{\tpr}\overline{\nu}(s)ds-(t+\tpr)\overline{\nu}(-t)\right)<0.
\end{align*}
\item[2.]Let $\left(B^{(\nu)}_u,u\geq0\right)$ be a standard Brownian motion such $B_0^{(\nu)}$ has law $\nu$, and let $\left(T_\mt,\mt\in\R_+^2\right)$ be the family of stopping times given by:
$$
\forall\,\mt\in\R_+^2,\quad T_\mt=\inf\left\{u\geq0,\,B_u^{(\nu)}\in]-\infty,-t]\cup[\tpr,+\infty[\right\}.
$$
Then $\left(B^{(\nu)}_{T_\mt},\mt\in\R_+^2\right)$ is a martingale associated to $\left(\mu_\mt^0,\mt\in\R_+^2\right)$ (see e.g.  \citep[Remark 4.2]{Ju16}). We provide another associated martingale to $\left(\mu_\mt^0,\mt\in\R_+^2\right)$ in Theorem \ref{theo:ExaNonMRL}.
\end{rem} 
One may observe that $\left(\mu_\mt^\ve,\mt\in\R_+^2\right)$ is a convex combination of processes obtained by censoring transformations. Indeed:
\begin{rem}
Consider the probability measures
$$
\nu^+=\frac{1}{\nu(]r,+\infty[)}1_{]r,+\infty[}\nu\quad\text{and}\quad \nu^-=\frac{1}{\nu(]-\infty,r])}1_{]-\infty,r]}\nu.
$$
\item[1.]We have
\begin{align*}
&\beta_\mt^-=\frac{\nu(]-\infty,r])}{t+\tpr}\int_{[r-t,r]}(y-r+t)\nu^-(dy),\text{ }\beta_\mt^+=\frac{\nu(]r,+\infty[)}{t+\tpr}\int_{]r,r+\tpr]}(y-r+t)\nu^+(dy),
\end{align*}
 $\alpha_\mt=\alpha^+_\mt+\alpha^-_\mt$, where
$$
\alpha^-_\mt=\frac{\nu(]-\infty,r])}{t+\tpr}\int_{[r-t,r]}(r+\tpr-y)\nu^-(dy) \text{ and } \alpha_\mt^+=\frac{\nu(]r,+\infty[)}{t+\tpr}\int_{]r,r+\tpr]}(r+\tpr-y)\nu^+(dy).
$$
\item[2.]Let $\left(\eta_\mt^\ve,\mt\in\R_+^2\right)$ and $\left(\sigma_\mt,\mt\in\R_+^2\right)$ be the processes given by:
\begin{equation}\label{eq:ExaNonMRLGleEta}
\left.
\begin{array}{ll}
&\forall\,t\in\R_+,\quad\eta_{(0,t)}^\ve=\nu^-\text{ and,}\\
&\forall\,\mt\in\R_+^\ast\times\R_+,\quad\eta_\mt^\ve=1_{]-\infty,r-t[}\nu^-+\alpha_\mt^{\star-}\delta_{r-t}+\beta_\mt^{\star-}\delta_{r+(1+\ve)\tpr}
\end{array}
\right\}
\end{equation}
and
\begin{equation}\label{eq:ExaNonMRLGleSigma}
\left.
\begin{array}{ll}
&\forall\,t\in\R_+,\quad\sigma_{(t,0)}=\nu^+\text{ and,}\\
&\forall\,\mt\in\R_+\times\R_+^\ast,\quad\sigma_\mt= \alpha_\mt^{\star+}\delta_{r-t}+\beta_\mt^{\star+}\delta_{r+\tpr}+1_{]r+\tpr,+\infty[}\nu^+
\end{array}
\right\}
\end{equation}
respectively, where
$$
\alpha_\mt^{\star-}=\frac{\alpha_\mt^-}{\nu(]-\infty,r])},\quad \alpha_\mt^{\star+}=\frac{\alpha_\mt^+}{\nu(]r,+\infty[)}, \quad \beta_\mt^{\star-}=\frac{\beta_\mt^-}{\nu(]-\infty,r])} \text{ and }\beta_\mt^{\star+}=\frac{\beta_\mt^+}{\nu(]r,+\infty[)}.
$$
Then, for every $\mt\in\R_+^2$, one has
$$
\mu_\mt^\ve=\nu(]-\infty,r])\eta_\mt^\ve+\nu(]r,+\infty[)\sigma_\mt.
$$
\end{rem}
\begin{prop}\label{prop:ExaNonMRLicx}
The process $\left(\mu^\ve_\mt,\mt\in\R_+^2\right)$ is ordered by the increasing convex dominance.
\end{prop}
\begin{proof}
It suffices to show that $\left(\eta_\mt^\ve,\mt\in\R_+^2\right)$ and $\left(\sigma_\mt,\mt\in\R_+^2\right)$ are ordered by the increasing convex dominance. It follows from Proposition \ref{prop:ExaMTP2ModifNu} that $\left(\eta_\mt^\ve,\mt\in\R_+^2\right)$ is a MRL process which, according to Theorem 4.A.26 in \citep{ShS07}, implies that it is ordered by the increasing convex dominance. Moreover, if $h_{\#}\nu^{+}$ and $h_{\#}\sigma_\mt$ denote the image of $\nu^+$ and $\sigma_\mt$ respectively under $h:\,x\longmapsto-x$, then, by Proposition \ref{prop:ExaMTP2ModifNu}, the process $\left(h_{\#}\sigma_\mt,\mt\in\R_+^2\right)$ given by
\begin{equation}\label{eq:MRLTwoDhdiez}
\left.
\begin{array}{ll}
&\forall\,t\in\R_+,\quad h_{\#}\sigma_{(t,0)}=h_{\#}\nu^{+}\text{ and, }\\
&\forall\,\mt\in\R_+\times\R_+^\ast,\quad h_{\#}\sigma_\mt=1_{]-\infty,-r-\tpr[}h_{\#}\nu^{+}+\beta^{\star+}_\mt\delta_{-r-\tpr}+\alpha^{\star+}\delta_{-r+t},
\end{array}
\right\}
\end{equation}
is MRL ordered. Then  $\left(h_{\#}\sigma_\mt,\mt\in\R_+^2\right)$ is also ordered by the increasing convex dominance (see Theorem 4.A.26 in  \citep{ShS07}). This means that $\left(\sigma_\mt,\mt\in\R_+^2\right)$ is ordered by the decreasing convex dominance. Hence, since $\left(\sigma_\mt,\mt\in\R_+^2\right)$ has constant mean, $\left(\sigma_\mt,\mt\in\R_+^2\right)$ is also ordered by the increasing convex dominance.
\end{proof}
Since Kellerer's theorem fails in the two-parameter case (see  \citep[Theorem 2.2]{Ju16}), Proposition \ref{prop:ExaNonMRLicx} is not sufficient for the existence of an associated submartingale to $\left(\mu_\mt^\ve,\mt\in\R_+^2\right)$. Moreover, $\left(\mu_\mt^\ve,\mt\in\R_+^2\right)$ is not MRL ordered and the Cox-Hobson algorithm does not apply. Nevertheless, $\left(\mu_\mt^\ve,\mt\in\R_+^2\right)$ is a convex combination of two processes to which  the Cox-Hobson algorithm provides associated  submartingales. Because a convex combination of submartingale measures is still a submartingale measure, we deduce that $\left(\mu_\mt^\ve,\mt\in\R_+^2\right)$ is associated to a submartingale.
\begin{theorem}\label{theo:ExaNonMRL}
The Cox-Hobson algorithm allows to construct an associated submartingale to the process $\left(\mu_\mt^\ve,\mt\in\R_+^2\right)$.
\end{theorem}
\begin{proof}
It suffices to show that the Cox-Hobson algorithm yields an associated submartingale to each of $\left(\eta_\mt^\ve,\mt\in\R_+^2\right)$
 and $\left(\sigma_\mt,\mt\in\R_+^2\right)$. Indeed, if $\left(M_\mt^{\eta,\ve},\mt\in\R_+^2\right)$ and $\left(M_\mt^{\sigma},\mt\in\R_+^2\right)$ denote the submartingles associated to $\left(\eta_\mt^\ve,\mt\in\R_+^2\right)$ and $\left(\sigma_\mt,\mt\in\R_+^2\right)$ respectively, and if $Y$ denotes a Bernoulli random variable independent of $\left(M_\mt^{\eta,\ve},\mt\in\R_+^2\right)$ and $\left(M_\mt^{\sigma},\mt\in\R_+^2\right)$ such that 
$$
\Pb(Y=1)=\nu(]-\infty,r])=1-\Pb(Y=0),
$$
then the process $\left(X_\mt^\ve,\mt\in\R_+^2\right)$ given by:
$$
\forall\,\mt\in\R_+^2,\quad X_\mt^\ve=YM_\mt^{\eta,\ve}+(1-Y)M_\mt^{\sigma}
$$
is associated to $\left(\mu_\mt^\ve,\mt\in\R_+^2\right)$ and, for every positive integer $n$, every $\ms_1\leq\cdots\leq\ms_n\leq\ms\leq\mt$ elements of $\R_+^2$ and every continuous bounded function $\Phi:\,\R^{n+1}\to\R$,
\begin{align*}
&\E\left[\Phi\left(X_{\ms_1}^\ve,\cdots,X_{\ms_n}^\ve,X_{\ms}^\ve\right)\left(X_\mt^\ve-X_\ms^\ve\right)\right]\\
&=\nu(]-\infty,r])\E\left[\Phi\left(M_{\ms_1}^{\eta,\ve},\cdots,M_{\ms_n}^{\eta,\ve},M_{\ms}^{\eta,\ve}\right)\left(M_\mt^{\eta,\ve}-M_\ms^{\eta,\ve}\right)\right]\\
&\quad+\nu(]r,+\infty[)\E\left[\Phi\left(M_{\ms_1}^{\sigma},\cdots,M_{\ms_n}^{\sigma},M_{\ms}^{\sigma}\right)\left(M_\mt^{\sigma}-M_\ms^{\sigma}\right)\right]\geq0
\end{align*}
which shows that $\left(X_\mt^\ve,\mt\in\R_+^2\right)$ is a submartingale.

Since $\left(\eta_\mt^\ve,\mt\in\R_+^2\right)$ is a MRL process, the Cox-Hobson algorithm provides an associated submartingale to it (see Remark \ref{rem:MRLTwoDSubMart}). Moreover, as the process $\left(h_{\#}\sigma_\mt,\mt\in\R_+^2\right)$, defined by (\ref{eq:MRLTwoDhdiez}), is MRL ordered, the Cox-Hobson algorithm also yields an associated submartingale $\left(M_\mt^{h_{\#}\sigma},\mt\in\R_+^2\right)$ to $\left(h_{\#}\sigma_\mt,\mt\in\R_+^2\right)$. Since $\left(h_{\#}\sigma_\mt,\mt\in\R_+^2\right)$ has constant mean, $\left(M_\mt^{h_{\#}\sigma},\mt\in\R_+^2\right)$ has also constant mean, and then it is a martingale. It remains to observe that $\left(-M_\mt^{h_{\#}\sigma},\mt\in\R_+^2\right)$ is a martingale associated to 
$\left(\sigma_\mt,\mt\in\R_+^2\right)$.
\end{proof}
\begin{rem}
The process $\left(\mu^0_\mt,\mt\in\R_+^2\right)$ given by (\ref{eq:ExaNonMRLMuzero}) is ordered by the convex dominance, and, by Theorem \ref{theo:ExaNonMRL}, the Cox-Hobson algorithm provides an associated martingale.
\end{rem}

\subsection{Processes obtained by subordination}
Now, we apply Theorem \ref{theo:MTP2TailSub} to exhibit other non-MRL processes each of which is associated to a submartingale.
\begin{theorem}
Let $\left(\mu_{(\lambda,\lambda')},(\lambda,\lambda')\in\R_+^2\right)$ be the process defined by (\ref{eq:ExaNonMRLMTP2}).
\item[1.]Let $\left(p_t,t\in\R_+^\ast\right)$ and $\left(q_t,t\in\R_+^\ast\right)$ be the transition densities of two right-continuous, integrable and time-homogeneous $\R_+$-valued Markov processes issued from zero. Suppose that, for every $t\in\R_+$, $p_t$ and $q_t$ satisfy (\ref{eq:TP2MarkovCS}). Then the Cox-Hobson algorithm allows to associate a submartingale to the process $\left(\zeta_\mt^\ve,\mt\in(\R_+^\ast)^2\right)$ given by:
$$
\forall\,(\mt,x)\in(\R_+^\ast)^2\times\R,\quad\zeta_\mt^\ve([x,+\infty[)=\int\!\!\!\int_{\R_+^2}\mu^\ve_{(\lambda,\lambda')}([x,+\infty[)p_t(0,\lambda)q_{\tpr}(0,\lambda')d\lambda d\lambda'.
$$
\item[2.] Let $\left(P_t,t\in\R_+\right)$ and $\left(Q_t,t\in\R_+\right)$ be the transition matrices of two right-continuous, integrable and time-homogeneous $\R_+$-valued Markov chain started at zero. If, for every $t\in\R_+$, $P_t$ and $Q_t$ fulfill Condition (\ref{eq:TP2MarkovChainCS}), then the Cox-Hobson algorithm yields an associated submartingale to the process $\left(\xi_\mt^\ve,\mt\in\R_+^2\right)$ defined by:
$$
\forall\,(\mt,x)\in\R_+^2\times\R,\quad\xi_\mt^\ve([x,+\infty[)=\sum\limits_{i\in\N}\sum\limits_{j\in\N}\mu_{(i,j)}^\ve([x,+\infty[)P_t(0,i)Q_{\tpr}(0,j).
$$
\end{theorem} 
\begin{proof}
We prove only Point 1 since the proof of Point 2 relies on the same arguments. Let $\left(\eta_\mt^\ve,\mt\in\R_+^2\right)$ and $\left(\sigma_\mt,\mt\in\R_+^2\right)$ be the processes defined by (\ref{eq:ExaNonMRLGleEta}) and (\ref{eq:ExaNonMRLGleSigma}) respectively. Let $\left(\widehat{\eta}_\mt^\ve,\mt\in(\R_+^\ast)^2\right)$ and $\left(\widehat{\sigma}_\mt,\mt\in(\R_+^\ast)^2\right)$ be the processes given, for every $\mt\in(\R_+^\ast)^2$ and $x\in\R$ by:
$$
\widehat{\eta}_\mt^\ve([x,+\infty[)=\int\!\!\!\int_{\R_+^2}\eta^\ve_{(\lambda,\lambda')}([x,+\infty[)p_t(0,\lambda)q_{\tpr}(0,\lambda')d\lambda d\lambda'
$$
and
$$
\widehat{\sigma}_\mt([x,+\infty[)=\int\!\!\!\int_{\R_+^2}\sigma_{(\lambda,\lambda')}([x,+\infty[)p_t(0,\lambda)q_{\tpr}(0,\lambda')d\lambda d\lambda'
$$
respectively. Then
$$
\zeta^\ve_\mt=\nu(]-\infty,r])\widehat{\eta}_\mt^\ve+\nu(]r,+\infty[)\widehat{\sigma}_\mt.
$$
The existence of an associated submartingale $\left(M_\mt^{\widehat{\eta},\ve},\mt\in(\R_+^\ast)^2\right)$ to $\left(\widehat{\eta}_\mt^\ve,\mt\in(\R_+^\ast)^2\right)$ is deduced from Point 1 of Theorem \ref{theo:MTP2TailSub}. On the other hand, the image $h_{\#}\widehat{\sigma}_\mt$ of $\widehat{\sigma}_\mt$ under $h:\,y\longmapsto -y$ is given by:
\begin{align*}
\forall\,x\in\R,\quad h_{\#}\widehat{\sigma}_{\mt}([x,+\infty[)=\int\!\!\!\int_{\R_+^2}h_{\#}\sigma_{(\lambda,\lambda')}([x,+\infty[)p_t(0,\lambda)q_{\tpr}(0,\lambda')\,d\lambda\,d\lambda',
\end{align*}
where $h_{\#}\sigma_{(\lambda,\lambda')}$ is the image of $\sigma_{(\lambda,\lambda')}$ under $h$. Since $\left(h_{\#}\sigma_\mt,\mt\in\R_+^2\right)$ has a MTP$_2$ integrated 
survival function, it follows from Point 1 of Theorem \ref{theo:MTP2TailSub} that $\left(h_{\#}\widehat{\sigma}_\mt,\mt\in\R_+^2\right)$ has also MTP$_2$ integrated survival function. In particular, $\left(h_{\#}\widehat{\sigma}_\mt,\mt\in\R_+^2\right)$ is a MRL process. Hence, the Cox-Hobson algorithm provides an associated martingale $\left(M_\mt^{h_{\#}\widehat{\sigma}},\mt\in(\R_+^\ast)^2\right)$ to $\left(h_{\#}\widehat{\sigma}_\mt,\mt\in(\R_+^\ast)^2\right)$. Then, as in the proof of Theorem \ref{theo:ExaNonMRL}, we deduce that the Cox-Hobson algorithm yields an associated submartingale to $\left(\zeta^\ve_\mt,\mt\in(\R_+^\ast)^2\right)$.
\end{proof}

\subsection{Processes obtained by convolution }
Using the Cox-Hobson algorithm, one may also associated submartingales to certain non-MRL processes obtained by convolution. For instance, if $\left(\mu_\mt^\ve,\mt\in\R_+^2\right)$ is the process defined by (\ref{eq:ExaNonMRLMTP2}) and if $f$ is a Lebesgue-integrable log-concave and positive density function, then one may associate a submartingale to the process
$$
\chi^\ve_\mt(dy)=\left(\int_{\R}f(y-z)\mu^\ve_\mt(dz)\right)dy.
$$
Indeed, for every $\mt\in\R_+^2$, $\mu_\mt^\ve=\nu(]-\infty,r])\eta_\mt^\ve+\nu(]r,+\infty[)\sigma_\mt$,
where the processes $\left(\eta_\mt^\ve,\mt\in\R_+^2\right)$ and $\left(\sigma_\mt,\mt\in\R_+^2\right)$ are defined by (\ref{eq:ExaNonMRLGleEta}) and (\ref{eq:ExaNonMRLGleSigma}) respectively. Hence,
\[
\chi_\mt^\ve=\nu(]-\infty,r])\chi_\mt^{\eta,\ve}+\nu(]r,+\infty[)\chi_\mt^{\sigma},
\]
where
\begin{align*}
\chi^{\eta,\ve}_\mt(dy)=\left(\int_{\R}f(y-z)\eta^\ve_\mt(dz)\right)dy
\text{ and }
\chi^{\sigma}_\mt(dy)=\left(\int_{\R}f(y-z)\sigma_\mt(dz)\right)dy.
\end{align*}
Then, applying Proposition \ref{prop:ExaMRLMTPConv} and the Cox-Hobson algorithm, one may construct an associated submartingale to  $\left(\chi^{\eta,\ve}_\mt,\mt\in\R_+^2\right)$ and an associated martingale 
to $\left(\chi^{\sigma}_\mt,\mt\in\R_+^2\right)$.
\noindent
\paragraph*{Acknowledgements}We are grateful to the anonymous referee for a careful reading and valuable  comments that led to a substantial improvement of the paper.


\end{document}